\documentclass[12pt]{article}
\usepackage{amsmath}
\usepackage{amsfonts}
\usepackage{amssymb}

\def\thebibliograph#1#2{\section*{{\normalsize \bf #2}}\list
   {[\arabic{enumi}]}{\settowidth\labelwidth{[#1]}\leftmargin\labelwidth
     \advance\leftmargin\labelsep
     \usecounter{enumi}}
     \def\newblock{\hskip .11em plus .33em minus -.07em}
     \sloppy
     \sfcode`\.=1000\relax}

\newtheorem{theorem}{Theorem}

\newtheorem{definition}{Definition}
\newtheorem{corollary}{Corollary}
\newtheorem{lemma}{Lemma}
\newtheorem{remark}{Remark}
\input{tcilatex}
\begin{document}

\title{Variable Triebel-Lizorkin-type spaces }
\author{ Douadi Drihem \ \thanks{%
M'sila University, Department of Mathematics, Laboratory of Functional
Analysis and Geometry of Spaces , P.O. Box 166, M'sila 28000, Algeria,
e-mail: \texttt{\ douadidr@yahoo.fr}}}
\date{\today }
\maketitle

\begin{abstract}
In this paper we study Triebel-Lizorkin-type spaces with variable smoothness
and integrability. We show that our space is well-defined, i.e., independent
of the choice of basis functions and we obtain their atomic
characterization. Moreover the Sobolev embeddings for these function spaces
are obtained.\vskip5pt

\textit{MSC 2010\/}: 46E35.

\textit{Key words and phrases}. Atom, Embeddings, Triebel space, maximal
function, variable exponent.
\end{abstract}

\section{Introduction}

Function spaces with variable exponents have been intensively studied in the
recent years by a significant number of authors. The motivation to study
such function spaces comes from applications to other fields of applied
mathematics, such that fluid dynamics and image processing. Some examples of
these spaces can be mentioned such as: variable exponent Lebesgue spaces $%
L^{p(\cdot )}$\ and variable Sobolev spaces $W^{k,p(\cdot )}$\ where the
study of these function spaces was initiated in \cite{KR91}. Almeida and
Samko \cite{AS06} and Gurka, Harjulehto and Nekvinda \cite{GHN07} introduced
variable exponent Bessel potential spaces, wich generalize the classical
Bessel potential spaces. Leopold [24-27] and Leopold \& Schrohe [28] studied
pseudo-differential operators, they introduced related Besov spaces with
variable smoothness $B_{p,p}^{\alpha (\cdot )}$. Along a different line of
study, J.-S. Xu \cite{Xu08}, \cite{Xu09} has studied Besov spaces with
variable $p$, but fixed $q$ and $\alpha $. More general function spaces with
variable smoothness were explicitly studied by Besov \cite{B03}, including
characterizations by differences. Besov spaces of variable smoothness and
integrability, $B_{p(\cdot ),q(\cdot )}^{\alpha (\cdot )}$, initially
appeared in the paper of A. Almeida and P. H\"{a}st\"{o} \cite{AH}. Several
basic properties were established, such as the Fourier analytical
characterisation and Sobolev embeddings. When $p,q,\alpha $ are constants
they coincide with the usual function spaces $B_{p,q}^{s}$. Also, A. I.
Tyulenev \cite{Ty151}, \cite{Ty152} has studied some new function spaces of
variable smoothness.

Triebel-Lizorkin spaces with variable exponents\ were introduced by \cite%
{DHR}. They proved a discretization by the so called $\varphi $-transform.
Also atomic and molecular decomposition of these function spaces are
obtained and used it to derive trace results. The Sobolev embedding of these
function spaces was proved by J. Vyb\'{\i}ral, \cite{V}. Some properties of
these function spaces such as local means characterizations and
characterizations by ball means of differences can be found in \cite{KV121}
and \cite{KV122}. When $\alpha ,p,q$ are constants they coincide with the
usual function spaces $F_{p,q}^{\alpha }$.\vskip5pt

In recent years, there has been increasing interest in a new family of
function spaces\ which generalize the\ Triebel-Lizorkin spaces, called $%
F_{p,q}^{s,\tau }$ spaces, were introduced and studied in \cite{YY1}. When $%
\tau =0$, they coincide with the usual function spaces $F_{p,q}^{s}$.
Various properties of these function spaces including atomic, molecular or
wavelet decompositions, characterizations by differences, have already been
established in [10, 12-13, 29, 33-35, 43-46, 48]. Moreover, these function
spaces, including some of their special cases related to $Q$ spaces.\vskip5pt

Based on Triebel-Lizorkin-type spaces and Triebel-Lizorkin spaces with
variable exponents $F_{p(\cdot ),q(\cdot )}^{\alpha (\cdot )}$, we present
another Triebel-Lizorkin-type spaces with variable smoothness and
integrability, which covers Triebel-Lizorkin-type spaces with fixed
exponents. These type of function spaces are introduced in \cite{YZW15},
where several properties are obtained such as atomic decomposition and the
boundedness of trace operator.\vskip5pt

\ \ The paper is organized as follows. First we give some preliminaries
where we fix some notations and recall some basics facts on function spaces
with variable integrability\ and we give some key technical lemmas needed in
the proofs of the main statements. For making the presentation clearer, we
give their proofs later in Section 6. In Section 3\ we define the spaces $%
\mathfrak{F}_{p(\cdot ),q(\cdot )}^{\alpha (\cdot ),\tau (\cdot )}$ where
several basic properties such as the $\varphi $-transform characterization
are obtained. In Section 4 we prove elementary embeddings between these
functions spaces, as well as Sobolev embeddings. In Section 5, we give the
atomic decomposition of these function spaces.

\section{Preliminaries}

As usual, we denote by $\mathbb{R}^{n}$ the $n$-dimensional real Euclidean
space, $\mathbb{N}$ the collection of all natural numbers and $\mathbb{N}%
_{0}=\mathbb{N}\cup \{0\}$. The letter $\mathbb{Z}$ stands for the set of
all integer numbers.\ The expression $f\lesssim g$ means that $f\leq c\,g$
for some independent constant $c$ (and non-negative functions $f$ and $g$),
and $f\approx g$ means $f\lesssim g\lesssim f$. As usual for any $x\in 
\mathbb{R}$, $[x]$ stands for the largest integer smaller than or equal to $%
x $.\vskip5pt

By supp $f$ we denote the support of the function $f$ , i.e., the closure of
its non-zero set. If $E\subset {\mathbb{R}^{n}}$ is a measurable set, then $%
|E|$ stands for the (Lebesgue) measure of $E$ and $\chi _{E}$ denotes its
characteristic function.\vskip5pt

The Hardy-Littlewood maximal operator $\mathcal{M}$ is defined on $L_{%
\mathrm{loc}}^{1}$ by%
\begin{equation*}
\mathcal{M}f(x)=\sup_{r>0}\frac{1}{\left\vert B(x,r)\right\vert }%
\int_{B(x,r)}\left\vert f(y)\right\vert dy
\end{equation*}%
and $M_{B}f(x)=$ $\frac{1}{\left\vert B\right\vert }\int_{B}\left\vert
f(y)\right\vert dy$, $x\in B$. The symbol $\mathcal{S}(\mathbb{R}^{n})$ is
used in place of the set of all Schwartz functions $\varphi $ on $\mathbb{R}%
^{n}$. We denote by $\mathcal{S}^{\prime }(\mathbb{R}^{n})$ the dual space
of all tempered distributions on $\mathbb{R}^{n}$. The Fourier transform of
a tempered distribution $f$ is denoted by $\mathcal{F}f$ while its inverse
transform is denoted by $\mathcal{F}^{-1}f$.\vskip5pt

For $v\in \mathbb{Z}$ and $m=(m_{1},...,m_{n})\in \mathbb{Z}^{n}$, let $%
Q_{v,m}$ be the dyadic cube in $\mathbb{R}^{n}$, $Q_{v,m}=%
\{(x_{1},...,x_{n}):m_{i}\leq 2^{v}x_{i}<m_{i}+1,i=1,2,...,n\}$. For the
collection of all such cubes we use $\mathcal{Q}:=\{Q_{v,m}:v\in \mathbb{Z}%
,m\in \mathbb{Z}^{n}\}$. For each cube $Q$, we denote its center by $c_{Q}$,
its lower left-corner by $x_{Q_{v,m}}=2^{-v}m$ of $Q=Q_{v,m}$ and its side
length by $l(Q)$. For $r>0$, we denote by $rQ$ the cube concentric with $Q$
having the side length $rl(Q)$. Furthermore, we put $v_{Q}=-\log _{2}l(Q)$
and $v_{Q}^{+}=\max (v_{Q},0)$.\vskip5pt

For $v\in \mathbb{Z}$, $\varphi \in \mathcal{S}(\mathbb{R}^{n})$ and $x\in 
\mathbb{R}^{n}$, we set $\widetilde{\varphi }(x):=\overline{\varphi (-x)}$, $%
\varphi _{v}(x):=2^{vn}\varphi (2^{v}x)$, and%
\begin{equation*}
\varphi _{v,m}(x):=2^{vn/2}\varphi (2^{v}x-m)=|Q_{v,m}|^{1/2}\varphi
_{v}(x-x_{Q_{v,m}})\quad \text{if\quad }Q=Q_{v,m}.
\end{equation*}

The variable exponents that we consider are always measurable functions $p$
on $\mathbb{R}^{n}$ with range in $[c,\infty \lbrack $ for some $c>0$. We
denote the set of such functions by $\mathcal{P}_{0}$. The subset of
variable exponents with range $[1,\infty \lbrack $ is denoted by $\mathcal{P}
$. We use the standard notation $p^{-}:=\underset{x\in \mathbb{R}^{n}}{\text{%
ess-inf}}$ $p(x)$,$\quad p^{+}:=\underset{x\in \mathbb{R}^{n}}{\text{ess-sup 
}}p(x)$.

The variable exponent Lebesgue space $L^{{p(\cdot )}}$ is the class of all
measurable functions $f$ on ${\mathbb{R}^{n}}$ such that the modular 
\begin{equation*}
\varrho _{{p(\cdot )}}(f):=\int_{\mathbb{R}^{n}}|f(x)|^{p(x)}\,dx
\end{equation*}%
is finite. This is a quasi-Banach function space equipped with the
quasi-norm 
\begin{equation*}
\Vert f\Vert _{p(\cdot )}:=\inf \Big\{\mu >0:\varrho _{{p(\cdot )}}\Big(%
\frac{1}{\mu }f\Big)\leq 1\Big\}.
\end{equation*}%
If $p(x):=p$ is constant, then $L^{{p(\cdot )}}=L^{p}$ is the classical
Lebesgue space.

A useful property is that $\varrho _{p(\cdot )}(f)\leqslant 1$ if and only
if $\Vert f\Vert _{p(\cdot )}\leqslant 1$ (\textit{unit ball property}).
This property is clear for constant exponents due to the obvious relation
between the norm and the modular in that case.

Let $p,q\in \mathcal{P}_{0}$. The space $L^{p(\cdot )}(\ell ^{q(\cdot )})$
is defined to be the set of all sequences $\left( f_{v}\right) _{v\geq 0}$
of functions such that%
\begin{equation*}
\left\Vert \left( f_{v}\right) _{v\geq 0}\right\Vert _{L^{p(\cdot )}(\ell
^{q(\cdot )})}:=\left\Vert \left\Vert \left( f_{v}(x)\right) _{v\geq
0}\right\Vert _{\ell ^{q(x)}}\right\Vert _{L^{p(\cdot )}}<\infty .
\end{equation*}%
It is easy to show that $L^{p(\cdot )}(\ell ^{q(\cdot )})$\ is always a
quasi-normed space\ and it is a normed space, if $\min (p(x),q(x))\geq 1$\
holds point-wise.

We say that $g:\mathbb{R}^{n}\rightarrow \mathbb{R}$ is \textit{locally }log%
\textit{-H\"{o}lder continuous}, abbreviated $g\in C_{\text{loc}}^{\log }$,
if there exists $c_{\log }(g)>0$ such that%
\begin{equation}
\left\vert g(x)-g(y)\right\vert \leq \frac{c_{\log }(g)}{\log
(e+1/\left\vert x-y\right\vert )}  \label{lo-log-Holder}
\end{equation}%
for all $x,y\in \mathbb{R}^{n}$. We say that $g$ satisfies the log\textit{-H%
\"{o}lder decay condition}, if there exists $g_{\infty }\in \mathbb{R}$ and
a constant $c_{\log }>0$ such that%
\begin{equation*}
\left\vert g(x)-g_{\infty }\right\vert \leq \frac{c_{\log }}{\log
(e+\left\vert x\right\vert )}
\end{equation*}%
for all $x\in \mathbb{R}^{n}$. We say that $g$ is \textit{globally}-log%
\textit{-H\"{o}lder continuous}, abbreviated $g\in C^{\log }$, if it is%
\textit{\ }locally log-H\"{o}lder continuous and satisfies the log-H\"{o}%
lder decay\textit{\ }condition.\textit{\ }The constants $c_{\log }(g)$ and $%
c_{\log }$ are called the \textit{locally }log\textit{-H\"{o}lder constant }%
and the log\textit{-H\"{o}lder decay constant}, respectively\textit{.} We
note that all functions $g\in C_{\text{loc}}^{\log }$ always belong to $%
L^{\infty }$.\vskip5pt

We define the following class of variable exponents%
\begin{equation*}
\mathcal{P}^{\mathrm{log}}:=\Big\{p\in \mathcal{P}:\frac{1}{p}\in C^{\log }%
\Big\},
\end{equation*}%
were introduced in $\mathrm{\cite[Section \ 2]{DHHMS}}$. We define $%
1/p_{\infty }:=\lim_{|x|\rightarrow \infty }1/p(x)$\ and we use the
convention $\frac{1}{\infty }=0$. Note that although $\frac{1}{p}$ is
bounded, the variable exponent $p$ itself can be unbounded. It was shown in $%
\mathrm{\cite{DHHR}}$\textrm{, }Theorem 4.3.8 that $\mathcal{M}:L^{p(\cdot
)}\rightarrow L^{p(\cdot )}$ is bounded if $p\in \mathcal{P}^{\mathrm{log}}$
and $p^{-}>1$, see also $\mathrm{\cite{DHHMS}}$, Theorem 1.2.\ Let $p\in 
\mathcal{P}^{\mathrm{log}}$, $\varphi \in L^{1}$ and $\Psi \left( x\right)
:=\sup_{\left\vert y\right\vert \geq \left\vert x\right\vert }\left\vert
\varphi \left( y\right) \right\vert $. We suppose that $\Psi \in L^{1}$.
Then it was proved in $\mathrm{\cite[Lemma \ 4.6.3]{DHHR}}$ that 
\begin{equation*}
\Vert \varphi _{\varepsilon }\ast f\Vert _{{p(\cdot )}}\leq c\Vert \Psi
\Vert _{{1}}\Vert f\Vert _{{p(\cdot )}}
\end{equation*}%
for all $f\in L^{p(\cdot )}$, where $\varphi _{\varepsilon }:=\frac{1}{%
\varepsilon ^{n}}\varphi \left( \frac{\cdot }{\varepsilon }\right) $. We
also refer to the papers $\mathrm{\cite{CFMP}}$ and $\mathrm{\cite{Di}}$%
\textrm{,} where various results on maximal function in variable Lebesgue
spaces were obtained.\vskip5pt

Recall that $\eta _{v,m}(x):=2^{nv}(1+2^{v}\left\vert x\right\vert )^{-m}$,
for any $x\in \mathbb{R}^{n}$, $v\in \mathbb{N}_{0}$ and $m>0$. Note that $%
\eta _{v,m}\in L^{1}$ when $m>n$ and that $\left\Vert \eta _{v,m}\right\Vert
_{1}=c_{m}$ is independent of $v$, where this type of function was
introduced in \cite{HN07} and \cite{DHHR}. By $c$ we denote generic positive
constants, which may have different values at different occurrences.

\subsection{Some technical lemmas}

In this subsection we present some results which are useful for us. The
following lemma is from \cite[Lemma 19]{KV122}, see also \cite[Lemma 6.1]%
{DHR}.

\begin{lemma}
\label{DHR-lemma}Let $\alpha \in C_{\mathrm{loc}}^{\log }$ and let $R\geq
c_{\log }(\alpha )$, where $c_{\log }(\alpha )$ is the constant from $%
\mathrm{\eqref{lo-log-Holder}}$ for $\alpha $. Then%
\begin{equation*}
2^{v\alpha (x)}\eta _{v,m+R}(x-y)\leq c\text{ }2^{v\alpha (y)}\eta
_{v,m}(x-y)
\end{equation*}%
with $c>0$ independent of $x,y\in \mathbb{R}^{n}$ and $v,m\in \mathbb{N}%
_{0}. $
\end{lemma}

The previous lemma allows us to treat the variable smoothness in many cases
as if it were not variable at all, namely we can move the term inside the
convolution as follows:%
\begin{equation*}
2^{v\alpha (x)}\eta _{v,m+R}\ast f(x)\leq c\text{ }\eta _{v,m}\ast
(2^{v\alpha (\cdot )}f)(x).
\end{equation*}

\begin{lemma}
\label{r-trick}Let $r,R,N>0$, $m>n$ and $\theta ,\omega \in \mathcal{S}%
\left( \mathbb{R}^{n}\right) $ with $\mathrm{supp}\mathcal{F}\omega \subset 
\overline{B(0,1)}$. Then there exists $c=c(r,m,n)>0$ such that for all $g\in 
\mathcal{S}^{\prime }\left( \mathbb{R}^{n}\right) $, we have%
\begin{equation}
\left\vert \theta _{R}\ast \omega _{N}\ast g\left( x\right) \right\vert \leq
c\text{ }\max \Big(1,\Big(\frac{N}{R}\Big)^{m}\Big)(\eta _{N,m}\ast
\left\vert \omega _{N}\ast g\right\vert ^{r}(x))^{1/r},\quad x\in \mathbb{R}%
^{n},  \label{r-trick-est}
\end{equation}%
where $\theta _{R}(\cdot )=R^{n}\theta (R\cdot )$, $\omega _{N}(\cdot
)=N^{n}\omega (N\cdot )$ and $\eta _{N,m}:=N^{n}(1+N\left\vert \cdot
\right\vert )^{-m}$.
\end{lemma}

The proof of this lemma is given in \cite{D6}.

We will make use of the following statement, see \cite{DHHMS}, Lemma 3.3\
and \cite{DHHR}, Theorem 3.2.4.

\begin{theorem}
\label{DHHR-estimate}Let $p\in \mathcal{P}^{\mathrm{log}}$. Then for every $%
m>0$ there exists $\beta \in \left( 0,1\right) $ only depending on $m$ and $%
c_{\mathrm{log}}\left( p\right) $ such that%
\begin{eqnarray*}
&&\Big(\frac{\beta }{\left\vert Q\right\vert }\int_{Q}\left\vert
f(y)\right\vert dy\Big)^{p\left( x\right) } \\
&\leq &\frac{1}{\left\vert Q\right\vert }\int_{Q}\left\vert f(y)\right\vert
^{p\left( y\right) }dy \\
&&+\min \left( \left\vert Q\right\vert ^{m},1\right) \Big(\left(
e+\left\vert x\right\vert \right) ^{-m}+\frac{1}{\left\vert Q\right\vert }%
\int_{Q}\left( e+\left\vert y\right\vert \right) ^{-m}dy\Big),
\end{eqnarray*}%
for every cube $\mathrm{(}$or ball$\mathrm{)}$ $Q\subset \mathbb{R}^{n}$,
all $x\in Q\subset \mathbb{R}^{n}$and all $f\in L^{p\left( \cdot \right)
}+L^{\infty }$\ with $\left\Vert f\right\Vert _{p\left( \cdot \right)
}+\left\Vert f\right\Vert _{\infty }\leq 1$.
\end{theorem}

Notice that in the proof of this theorem we need only 
\begin{equation*}
\int_{Q}\left\vert f(y)\right\vert ^{p\left( y\right) }dy\leq 1
\end{equation*}%
and/or $\left\Vert f\right\Vert _{\infty }\leq 1$. We set%
\begin{equation*}
\left\Vert \left( f_{v}\right) _{v}\right\Vert _{L_{p(\cdot )}^{\tau (\cdot
)}\left( \ell ^{q(\cdot )}\right) }:=\sup_{P\in \mathcal{Q}}\Big\|\Big(\frac{%
f_{v}}{|P|^{\tau (\cdot )}}\chi _{P}\Big)_{v\geq v_{P}^{+}}\Big\|%
_{L^{p(\cdot )}(\ell ^{q(\cdot )})},
\end{equation*}%
where, $v_{P}=-\log _{2}l(P)$ and $v_{P}^{+}=\max (v_{P},0)$. The following
result is from \cite[Theorem 3.2]{DHR}.

\begin{theorem}
\label{DHR-theorem}Let $p,q\in C_{\mathrm{loc}}^{\log }$ with $1<p^{-}\leq
p^{+}<\infty $ and $1<q^{-}\leq q^{+}<\infty $. Then the inequality%
\begin{equation*}
\left\Vert (\eta _{v,m}\ast f_{v})_{v}\right\Vert _{L^{p(\cdot )}(\ell
^{q(\cdot )})}\leq c\left\Vert (f_{v})_{v}\right\Vert _{L^{p(\cdot )}(\ell
^{q(\cdot )})}.
\end{equation*}%
holds for every sequence $(f_{v})_{v\in \mathbb{N}_{0}}$ of $L_{\mathrm{loc}%
}^{1}$-functions and constant $m>n+c_{\log }(1/q)$.
\end{theorem}

We will make use of the following statement (we use it, since the maximal
operator is in general not bounded on $L^{p(\cdot )}(\ell ^{q(\cdot )})$,
see \cite[Section 5]{DHR}).

\begin{lemma}
\label{DHR-lemma1}Let $\mathbb{\tau }\in C_{\mathrm{loc}}^{\log }$, $\tau
^{-}>0$, $p,q\in \mathcal{P}^{\log }$\ with\ $0<\frac{\tau ^{+}}{\tau ^{-}}%
<\min \left( p^{-},q^{-}\right) $. For $m$ large enough sucth that 
\begin{equation*}
m>n\tau ^{+}+2n+w,\text{ }w>n+c_{\log }(1/q)+c_{\log }(\tau )
\end{equation*}%
and every sequence $(f_{v})_{v\in \mathbb{N}_{0}}$ of $L_{\mathrm{loc}}^{1}$%
-functions, there exists $c>0$\ such that%
\begin{equation*}
\left\Vert (\eta _{v,m}\ast f_{v})_{v}\right\Vert _{L_{p(\cdot )}^{\tau
(\cdot )}(\ell ^{q(\cdot )})}\leq c\left\Vert (f_{v})_{v}\right\Vert
_{L_{p(\cdot )}^{\tau (\cdot )}(\ell ^{q(\cdot )})}.
\end{equation*}
\end{lemma}

The proof of this lemma is postponed to the Appendix.

Let $\widetilde{L_{\tau (\cdot )}^{p(\cdot )}}$ be the collection of
functions $f\in L_{\text{loc}}^{p(\cdot )}(\mathbb{R}^{n})$ such that%
\begin{equation*}
\left\Vert f\right\Vert _{\widetilde{L_{\tau (\cdot )}^{p(\cdot )}}}:=\sup %
\Big\|\frac{f\chi _{P}}{|P|^{\tau (\cdot )}}\Big\|_{p(\cdot )}<\infty ,\quad
p\in \mathcal{P}_{0},\quad \tau :\mathbb{R}^{n}\rightarrow \mathbb{R}^{+},
\end{equation*}%
where the supremum is taken over all dyadic cubes $P$ with $|P|\geq 1$.
Notice that 
\begin{equation}
\left\Vert f\right\Vert _{\widetilde{L_{\tau (\cdot )}^{p(\cdot )}}}\leq
1\Leftrightarrow \sup_{P\in \mathcal{Q},|P|\geq 1}\Big\|\Big|\frac{f}{%
|P|^{\tau (\cdot )}}\Big|^{q(\cdot )}\chi _{P}\Big\|_{p(\cdot )/q(\cdot
)}\leq 1.  \label{mod-est}
\end{equation}

Let $\theta _{v}=2^{vn}\theta \left( 2^{v}\cdot \right) $. The following
lemma is from \cite{D7}.

\begin{lemma}
\label{key-estimate1}Let $v\in \mathbb{Z}$, $\mathbb{\tau }\in C_{\mathrm{loc%
}}^{\log }$, $\tau ^{-}>0$, $p\in \mathcal{P}_{0}^{\log }$, $0<r<p^{-}$ and $%
\theta ,\omega \in \mathcal{S}(\mathbb{R}^{n})$ with $\mathrm{supp}\mathcal{F%
}\omega \subset \overline{B(0,1)}$. For any $f\in \mathcal{S}^{\prime }(%
\mathbb{R}^{n})$ and any dyadic cube $P$ with $|P|\geq 1$, we have%
\begin{equation*}
\Big\|\frac{\theta _{v}\ast \omega _{v}\ast f}{|P|^{\tau (\cdot )}}\chi _{P}%
\Big\|_{p(\cdot )}\leq c\left\Vert \omega _{v}\ast f\right\Vert _{\widetilde{%
L_{\tau (\cdot )}^{p(\cdot )}}},
\end{equation*}%
such that the right-hand side is finite, where $c>0$ is independent of $v$
and $l(P).$
\end{lemma}

The next three lemmas are from \cite{DHR} where the first tells us that in
most circumstances two convolutions are as good as one.

\begin{lemma}
\label{Conv-est}For $v_{0},v_{1}\in \mathbb{N}_{0}$ and $m>n$, we have%
\begin{equation*}
\eta _{v_{0},m}\ast \eta _{v_{1},m}\approx \eta _{\min (v_{0},v_{1}),m}
\end{equation*}%
with the constant depending only on $m$ and $n$.
\end{lemma}

\begin{lemma}
\label{Conv-est1}Let $v\in\mathbb{N}_{0}$ and $m>n$. Then for any $Q\in 
\mathcal{Q}$ with $l(Q) =2^{-v}$, $y\in Q$ and $x\in \mathbb{R}^{n}$, we have%
\begin{equation*}
\eta _{v,m}\ast \left( \frac{\chi _{Q}}{|Q|}\right) (x)\approx \eta
_{v,m}(x-y)
\end{equation*}%
with the constant depending only on $m$ and $n$.
\end{lemma}

\begin{lemma}
\label{Conv-est2}Let $v,j\in \mathbb{N}_{0}$, $r\in (0,1]$ and $m>\frac{n}{r}
$. Then for any $Q\in \mathcal{Q}$ with $l(Q)=2^{-v}$, we have%
\begin{equation*}
(\eta _{j,m}\ast \eta _{v,m}\ast \chi _{Q})^{r}\approx
2^{(v-j)^{+}n(1-r)}\eta _{j,mr}\ast \eta _{v,mr}\ast \chi _{Q},
\end{equation*}%
where the constant depends only on $m$, $n$ and $r$.
\end{lemma}

The next lemma is a Hardy-type inequality which is easy to prove.

\begin{lemma}
\label{lq-inequality}\textit{Let }$0<a<1,J\in \mathbb{Z}$\textit{\ and }$%
0<q\leq \infty $\textit{. Let }$(\varepsilon _{k})_{k}$\textit{\ be a
sequences of positive real numbers and denote} $\delta
_{k}=\sum_{j=J^{+}}^{k}a^{k-j}\varepsilon _{j}$, $k\geq J^{+}$.\textit{\ }%
Then there exists constant $c>0\ $\textit{depending only on }$a$\textit{\
and }$q$ such that%
\begin{equation*}
\Big(\sum\limits_{k=J^{+}}^{\infty }\delta _{k}^{q}\Big)^{1/q}\leq c\text{ }%
\Big(\sum\limits_{k=J^{+}}^{\infty }\varepsilon _{k}^{q}\Big)^{1/q}.
\end{equation*}
\end{lemma}

\begin{lemma}
\label{Key-lemma}Let $\alpha ,\mathbb{\tau }\in C_{\mathrm{loc}}^{\log }$, $%
\mathbb{\tau }^{-}\geq 0$ and $p,q\in \mathcal{P}_{0}^{\log }$ with $%
0<q^{-}\leq q^{+}<\infty $. Let $(f_{k})_{k\in \mathbb{N}_{0}}$ be a
sequence of measurable functions on $\mathbb{R}^{n}$. For all $v\in \mathbb{N%
}_{0}$ and $x\in \mathbb{R}^{n}$, let $g_{v}(x)=\sum_{k=0}^{\infty
}2^{-|k-v|\delta }f_{k}(x)$. Then there exists a positive constant $c$,
independent of $(f_{k})_{k\in \mathbb{N}_{0}}$ such that%
\begin{equation*}
\left\Vert (g_{v})_{v}\right\Vert _{L_{p(\cdot )}^{\tau (\cdot )}\left( \ell
^{q(\cdot )}\right) }\leq c\left\Vert (f_{v})_{v}\right\Vert _{L_{p(\cdot
)}^{\tau (\cdot )}(\ell ^{q(\cdot )})},\quad \delta >\tau ^{+}.
\end{equation*}
\end{lemma}

The proof of Lemma \ref{Key-lemma} can be obtained by the same arguments
used in \cite[Lemma 2.10]{D6}.

The following statement can be found in \cite{BM}, that plays an essential
role later on.

\begin{lemma}
\label{BM-lemma}Let real numbers $s_{1}<s_{0}$ be given, and $\sigma \in
]0,1[$. For $0<q\leq \infty $ and $J\in \mathbb{N}_{0}$ there is $c>0$ such
that%
\begin{equation*}
\left\Vert (2^{\left( \sigma s_{0}+\left( 1-\sigma \right) s_{1}\right)
j}a_{j})_{j\geq J}\right\Vert _{\ell ^{_{q}}}\leq c\left\Vert
(2^{s_{0}j}a_{j})_{j\geq J}\right\Vert _{\ell ^{_{\infty }}}^{\sigma
}\left\Vert (2^{s_{1}j}a_{j})_{j\geq J}\right\Vert _{\ell ^{_{\infty
}}}^{1-\sigma }
\end{equation*}%
holds for all complex sequences $(2^{s_{0}j}a_{j})_{j\in \mathbb{N}_{0}}$ in 
$\ell ^{_{\infty }}$.
\end{lemma}

\section{The spaces\textbf{\ }$\mathfrak{F}_{p(\cdot ),q(\cdot )}^{\protect%
\alpha (\cdot ),\protect\tau (\cdot )}$}

In this section we\ present the Fourier analytical definition of
Triebel-Lizorkin-type spaces of variable smoothness and integrability\ and
we prove the basic properties in analogy to the Triebel-Lizorkin-type spaces
with fixed exponents. Select a pair of Schwartz functions $\Phi $ and $%
\varphi $ satisfy%
\begin{equation}
\text{supp}\mathcal{F}\Phi \subset \overline{B(0,2)}\text{ and }|\mathcal{F}%
\Phi (\xi )|\geq c\text{ if }|\xi |\leq \frac{5}{3}  \label{Ass1}
\end{equation}%
and 
\begin{equation}
\text{supp}\mathcal{F}\varphi \subset \overline{B(0,2)}\backslash B(0,1/2)%
\text{ and }|\mathcal{F}\varphi (\xi )|\geq c\text{ if }\frac{3}{5}\leq |\xi
|\leq \frac{5}{3}  \label{Ass2}
\end{equation}%
where $c>0$. Let $\alpha :\mathbb{R}^{n}\rightarrow \mathbb{R}$, $p,q,\tau
\in \mathcal{P}_{0}$ and $\Phi $ and $\varphi $ satisfy $\mathrm{\eqref{Ass1}%
}$ and $\mathrm{\eqref{Ass2}}$, respectively and we put $\varphi
_{v}=2^{vn}\varphi (2^{v}\cdot )$.

\begin{definition}
\label{B-F-def}Let $\alpha :\mathbb{R}^{n}\rightarrow \mathbb{R}$, $\tau :%
\mathbb{R}^{n}\rightarrow \mathbb{R}^{+}$, $p,q\in \mathcal{P}_{0}$ and $%
\Phi $ and $\varphi $ satisfy $\mathrm{\eqref{Ass1}}$ and $\mathrm{%
\eqref{Ass2}}$, respectively and we put $\varphi _{v}=2^{vn}\varphi
(2^{v}\cdot )$. The Triebel-Lizorkin-type space $\mathfrak{F}_{p(\cdot
),q(\cdot )}^{\alpha (\cdot ),\tau (\cdot )}$\ is the collection of all $%
f\in \mathcal{S}^{\prime }(\mathbb{R}^{n})$\ such that 
\begin{equation}
\left\Vert f\right\Vert _{\mathfrak{F}_{p(\cdot ),q(\cdot )}^{\alpha (\cdot
),\tau (\cdot )}}:=\sup_{P\in \mathcal{Q}}\Big\|\Big(\frac{2^{v\alpha \left(
\cdot \right) }\varphi _{v}\ast f}{|P|^{\tau (\cdot )}}\chi _{P}\Big)_{v\geq
v_{P}^{+}}\Big\|_{L^{p(\cdot )}(\ell ^{q(\cdot )})}<\infty ,  \label{B-def}
\end{equation}%
where $\varphi _{0}$ is replaced by $\Phi $.
\end{definition}

Using the system $(\varphi _{v})_{v\in \mathbb{N}_{0}}$ we can define the
norm%
\begin{equation*}
\left\Vert f\right\Vert _{F_{p,q}^{\alpha ,\tau }}:=\sup_{P\in \mathcal{Q}}%
\frac{1}{\left\vert P\right\vert ^{\tau }}\Big\|\Big(\sum%
\limits_{v=v_{P}^{+}}^{\infty }2^{v\alpha q}\left\vert \varphi _{v}\ast
f\right\vert ^{q}\chi _{P}\Big)|^{1/q}\Big\|_{p}
\end{equation*}%
for constants $\alpha $ and $p,q\in (0,\infty ]$. The Triebel-Lizorkin-type
space $F_{p,q}^{\alpha ,\tau }$ consist of all distributions $f\in \mathcal{S%
}^{\prime }(\mathbb{R}^{n})$ for which $\left\Vert f\right\Vert
_{F_{p,q}^{\alpha ,\tau }}<\infty $. It is well-known that these spaces do
not depend on the choice of the system $(\varphi _{v})_{v\in \mathbb{N}_{0}}$
(up to equivalence of quasinorms). Further details on the classical theory
of these spaces can be found in \cite{D1} and \cite{WYY}; see also \cite{D4}
for recent developments.

One recognizes immediately that if $\alpha $, $\tau $, $p$ and $q$ are
constants, then $\mathfrak{F}_{p(\cdot ),q(\cdot )}^{\alpha (\cdot ),\tau
(\cdot )}=F_{p,q}^{\alpha ,\tau }$. When, $q:=\infty $\ the
Triebel-Lizorkin-type space $\mathfrak{F}_{p(\cdot ),\infty }^{\alpha (\cdot
),\tau (\cdot )}$\ consist of all distributions $f\in \mathcal{S}^{\prime }(%
\mathbb{R}^{n})$\ such that 
\begin{equation*}
\Big\|\sup_{P\in \mathcal{Q},v\geq v_{P}^{+}}\frac{2^{v\alpha \left( \cdot
\right) }\varphi _{v}\ast f}{|P|^{\tau (\cdot )}}\chi _{P}\Big\|_{p(\cdot
)}<\infty .
\end{equation*}%
Let $B_{J}$ be any ball of $\mathbb{R}^{n}$ with radius $2^{-J}$, $J\in 
\mathbb{Z}$. In the definition of the spaces $\mathfrak{F}_{p(\cdot
),q(\cdot )}^{\alpha (\cdot ),\tau (\cdot )}$ if we replace the dyadic cubes 
$P$ by the balls $B_{J}$, then we obtain equivalent quasi-norms. From these
if we replace dyadic cubes $P$ in Definition \ref{B-F-def} by arbitrary
cubes $P$, we then obtain equivalent quasi-norms.

The Triebel-Lizorkin space of variable smoothness and integrability $%
F_{p(\cdot ),q(\cdot )}^{\alpha (\cdot )}$ is the collection of all $f\in 
\mathcal{S}^{\prime }(\mathbb{R}^{n})$\ such that 
\begin{equation*}
\left\Vert f\right\Vert _{F_{p(\cdot ),q(\cdot )}^{\alpha (\cdot
)}}:=\left\Vert \left( 2^{v\alpha \left( \cdot \right) }\varphi _{v}\ast
f\right) _{v\geq 0}\right\Vert _{L^{p(\cdot )}(\ell ^{q(\cdot )})}<\infty ,
\end{equation*}%
which introduced and investigated in \cite{DHR}, see \cite{KV122}\ and \cite%
{IN14} for further results. Obviously, $\mathfrak{F}_{p(\cdot ),q(\cdot
)}^{\alpha (\cdot ),0}=F_{p(\cdot ),q(\cdot )}^{\alpha (\cdot )}$. We refer
the reader to the recent paper \cite{YHSY} for further details, historical
remarks and more references on embeddings of Triebel-Lizorkin-type spaces
with fixed exponents. More information about Triebel-Lizorkin spaces can be
found in [35-38].

Sometimes it is of great service if one can restrict sup$_{P\in \mathcal{Q}}$
in the definition to a supremum taken with respect to dyadic cubes with side
length $\leq 1$.

\begin{lemma}
\label{new-equinorm}Let $\alpha ,\mathbb{\tau }\in C_{\mathrm{loc}}^{\log }$%
, $\mathbb{\tau }^{-}\geq 0$ and $p,q\in \mathcal{P}_{0}^{\log }$ with $%
\left( \mathbb{\tau }-\frac{1}{p}\right) ^{-}\geq 0$ and $%
0<q^{+},p^{+}<\infty $. A tempered distribution $f$ belongs to $\mathfrak{F}%
_{p(\cdot ),q(\cdot )}^{\alpha (\cdot ),\tau (\cdot )}$ if and only if,%
\begin{equation*}
\left\Vert f\right\Vert _{\mathfrak{F}_{p(\cdot ),q(\cdot )}^{\alpha (\cdot
),\tau (\cdot )}}^{\#}:=\sup_{P\in \mathcal{Q},|P|\leq 1}\Big\|\Big(\frac{%
2^{v\alpha \left( \cdot \right) }\varphi _{v}\ast f}{|P|^{\tau (\cdot )}}%
\chi _{P}\Big)_{v\geq v_{P}}\Big\|_{L^{p(\cdot )}(\ell ^{q(\cdot )})}<\infty
.
\end{equation*}%
Furthermore, the quasi-norms $\left\Vert f\right\Vert _{\mathfrak{F}%
_{p(\cdot ),q(\cdot )}^{\alpha (\cdot ),\tau (\cdot )}}$ and $\left\Vert
f\right\Vert _{\mathfrak{F}_{p(\cdot ),q(\cdot )}^{\alpha (\cdot ),\tau
(\cdot )}}^{\#}$ are equivalent.
\end{lemma}

\begin{proof}
Let $P$ be a dyadic cube such that $|P|=2^{-Jn}$, for some $-J\in \mathbb{N}$%
. Let $\{Q_{m}$ $:m=1,...,2^{-Jn}\}$ be the collection of all dyadic cubes
with volume $1$ and such that $P=\cup _{m=1}^{2^{-Jn}}Q_{m}$. Due to the
homogeneity, we may assume that $\left\Vert f\right\Vert _{\mathfrak{F}%
_{p(\cdot ),q(\cdot )}^{\alpha (\cdot ),\tau (\cdot )}}^{\#}=1$. It suffices
to show that%
\begin{equation*}
\int_{P}\Big(\sum\limits_{v=v_{P}^{+}}^{\infty }2^{v\alpha (x)q(x)}\frac{%
\left\vert \varphi _{v}\ast f(x)\right\vert ^{q(x)}}{|P|^{\tau (x)q(x)}}\Big)%
^{p(x)/q(x)}dx\lesssim 1
\end{equation*}%
for any dyadic cube $P$, with $|P|\geq 1$. The left-hand side can be
rewritten us%
\begin{eqnarray*}
&&\sum_{m=1}^{2^{-Jn}}\int_{Q_{m}}\Big(\sum\limits_{v=v_{P}^{+}}^{\infty
}2^{v\alpha (x)q(x)}\frac{\left\vert \varphi _{v}\ast f(x)\right\vert ^{q(x)}%
}{|P|^{\tau (x)q(x)}}\Big)^{p(x)/q(x)}dx \\
&\leq &2^{Jn}\sum_{m=1}^{2^{-Jn}}\int_{Q_{m}}\Big(\sum%
\limits_{v=v_{Q_{m}}^{+}}^{\infty }2^{v\alpha (x)q(x)}\frac{\left\vert
\varphi _{v}\ast f(x)\right\vert ^{q(x)}}{|Q_{m}|^{\tau (x)q(x)}}\Big)%
^{p(x)/q(x)}dx \\
&\lesssim &2^{Jn}\sum_{m=1}^{2^{-Jn}}1=1,
\end{eqnarray*}%
where we used $|P|^{p(x)\tau (x)}\geq 2^{-Jn}$, which completes the proof.
\end{proof}

\begin{remark}
We like to point out that this result with fixed exponents is given in \cite[%
Lemma 2.2]{WYY}.
\end{remark}

\begin{theorem}
Let $\alpha ,\mathbb{\tau }\in C_{\mathrm{loc}}^{\log }$, $\mathbb{\tau }%
^{-}\geq 0$ and $p,q\in \mathcal{P}_{0}^{\log }$ with $0<q^{+},p^{+}<\infty $%
. If $\left( \mathbb{\tau }-\frac{1}{p}\right) ^{-}>0$ or $\left( \mathbb{%
\tau }-\frac{1}{p}\right) ^{-}\geq 0$ and $q:=\infty $, then%
\begin{equation*}
\mathfrak{F}_{p(\cdot ),q(\cdot )}^{\alpha (\cdot ),\tau (\cdot )}=B_{\infty
,\infty }^{\alpha (\cdot )+n(\tau (\cdot )-1/p(\cdot ))}
\end{equation*}%
with equivalent quasi-norms.
\end{theorem}

\begin{proof}
We consider only $\left( \mathbb{\tau }-\frac{1}{p}\right) ^{-}>0$. The case 
$\left( \mathbb{\tau }-\frac{1}{p}\right) ^{-}\geq 0$ and $q:=\infty $ can
be proved analogously with the necessary modifications. Since $\left( 
\mathbb{\tau }-\frac{1}{p}\right) ^{-}>0$, then we use the equivalent norm
given in the previous lemma. First let us prove the following estimate%
\begin{equation*}
\left\Vert f\right\Vert _{\mathfrak{F}_{p(\cdot ),q(\cdot )}^{\alpha (\cdot
),\tau (\cdot )}}\lesssim \left\Vert f\right\Vert _{B_{\infty ,\infty
}^{\alpha (\cdot )+n(\tau (\cdot )-1/p(\cdot ))}}
\end{equation*}%
for any $f\in B_{\infty ,\infty }^{\alpha (\cdot )+n(\tau (\cdot )-1/p(\cdot
))}$. Let $P$ be a dyadic cube with volume $2^{-nv_{P}}$, $v_{P}\in \mathbb{N%
}$. We obtain that%
\begin{eqnarray*}
\frac{2^{v\alpha (x)}|\varphi _{v}\ast f(x)|}{|P|^{\tau (x)}} &\leq &c\text{ 
}2^{v(\alpha (x)+n(\tau (x)-1/p(x))+n(v_{P}-v)(\tau
(x)-1/p(x))+nv_{P}/p(x)}|\varphi _{v}\ast f(x)| \\
&\leq &c\text{ }2^{n(v_{P}-v)(\tau (x)-1/p(x))+nv_{P}/p(x)}\left\Vert
f\right\Vert _{B_{\infty ,\infty }^{\alpha (\cdot )+n(\tau (\cdot
)-1/p(\cdot ))}}
\end{eqnarray*}%
for any $x\in P$. Then for any $v\geq v_{P}$%
\begin{eqnarray*}
&&\Big\|\Big(\dsum\limits_{v=v_{P}}^{\infty }\left\vert \frac{2^{v\alpha
(\cdot )}\varphi _{v}\ast f}{|P|^{\tau (\cdot )}}\right\vert ^{q(\cdot
)}\chi _{P}\Big)^{1/q(\cdot )}\Big\|_{p(\cdot )} \\
&\lesssim &\left\Vert f\right\Vert _{B_{\infty ,\infty }^{\alpha (\cdot
)+n(\tau (\cdot )-1/p(\cdot ))}}\Big\|\Big(\dsum\limits_{v=v_{P}}^{\infty
}\left\vert 2^{n(v_{P}-v)(\tau (\cdot )-1/p(\cdot ))+nv_{P}/p(\cdot
)}\right\vert ^{q(\cdot )}\Big)^{1/q(\cdot )}\chi _{P}\Big\|_{p(\cdot )} \\
&\lesssim &c\text{ }\left\Vert f\right\Vert _{B_{\infty ,\infty }^{\alpha
(\cdot )+n(\tau (\cdot )-1/p(\cdot ))}}\left\Vert 2^{nv_{P}/p(\cdot )}\chi
_{P}\right\Vert _{\frac{p(\cdot )}{q(\cdot )}}\lesssim 1,
\end{eqnarray*}%
since $\left( \mathbb{\tau }-\frac{1}{p}\right) ^{-}>0$. Let $f\in \mathfrak{%
F}_{p(\cdot ),q(\cdot )}^{\alpha (\cdot ),\tau (\cdot )}$, with $\left\Vert
f\right\Vert _{\mathfrak{F}_{p(\cdot ),q(\cdot )}^{\alpha (\cdot ),\tau
(\cdot )}}=1$. By Lemma \ref{r-trick} we have for any $x\in \mathbb{R}%
^{n},m>n$ 
\begin{eqnarray*}
&&2^{v(\alpha (x)+n(\tau (x)-1/p(x))}|\varphi _{v}\ast f(x)| \\
&\leq &c\text{ }2^{v(\alpha (x)+n(\tau (x)-1/p(x))}(\eta _{v,m}\ast |\varphi
_{v}\ast f|^{p^{-}}(x))^{1/p^{-}} \\
&\leq &c\left\Vert 2^{v(\alpha (x)+n\tau (x))}\varphi _{v}\ast f(\cdot
)(1+2^{v}|x-\cdot |)^{-m/2p^{-}}\right\Vert _{p(\cdot )}\left\Vert
2^{vn/t(x)}(1+2^{v}|x-\cdot |)^{-m/2p^{-}}\right\Vert _{t(\cdot )},
\end{eqnarray*}%
by H\"{o}lder's inequality, with $\frac{1}{p^{-}}=\frac{1}{p(\cdot )}+\frac{1%
}{t(\cdot )}$. The second norm on the right-hand side is bounded if $m>\frac{%
2p^{-}}{t^{-}}(n+c_{\log }(1/t))$ (this is possible since $m$ can be taken
large enough). To show that, we investigate the corresponding modular:%
\begin{eqnarray*}
\varrho _{t(\cdot )}(2^{vn/t(x)}(1+2^{v}|x-\cdot |)^{-m/2p^{-}}) &=&\int_{%
\mathbb{R}^{n}}2^{vnt(y)/t(x)}(1+2^{v}|x-y|)^{-mt(y)/2p^{-}}dy \\
&\leq &2^{vn}\int_{\mathbb{R}^{n}}(1+2^{v}|x-y|)^{-(m-c_{\log
}(1/t))t^{-}/2p^{-}}dy<\infty ,
\end{eqnarray*}%
where we used Lemma \ref{DHR-lemma}. Again by the same lemma the first norm
is bounded by%
\begin{equation*}
\left\Vert 2^{v(\alpha (\cdot )+n\tau (\cdot ))}\varphi _{v}\ast f(\cdot
)(1+2^{v}|x-\cdot |)^{-h}\right\Vert _{p(\cdot )},
\end{equation*}%
where $h=\frac{m}{2p^{-}}-c_{\log }(\alpha +n\tau )$. Let now prove that
this expression is bounded. We investigate the corresponding modular:%
\begin{eqnarray}
&&\varrho _{p(\cdot )}(2^{v(\alpha (\cdot )+n\tau (\cdot ))}\varphi _{v}\ast
f(\cdot )(1+2^{v}|x-\cdot |)^{-h})  \notag \\
&=&\int_{\mathbb{R}^{n}}2^{v(\alpha (y)+n\tau (y))p(y)}|\varphi _{v}\ast
f(y)|^{p(y)}(1+2^{v}|x-y|)^{-hp(y)}dy  \notag \\
&=&\int_{|y-x|<2^{-v}}(\cdot \cdot \cdot )dy+\sum_{i=1}^{\infty
}\int_{2^{i-v}\leq |y-x|<2^{i-v+1}}(\cdot \cdot \cdot )dy  \notag \\
&\leq &\sum_{i=0}^{\infty }2^{-ihp^{-}}\int_{|y-x|<2^{i-v+1}}2^{v(\alpha
(y)+n\tau (y))p(y)}|\varphi _{v}\ast f(y)|^{p(y)}dy.  \label{est}
\end{eqnarray}%
Then the right-hand side of $\mathrm{\eqref{est}}$ is bounded by%
\begin{eqnarray*}
&&\sum_{i=0}^{\infty }2^{(n\tau ^{+}p^{+}-hp^{-})i}\int_{B(x,2^{i-v+1})}%
\frac{2^{v\alpha (y)p(y)}|\varphi _{v}\ast f(y)|^{p(y)}}{\left\vert
B(x,2^{i-v+1})\right\vert ^{p(y)\tau (y)}}dy \\
&\leq &\sum_{i=0}^{\infty }2^{(n\tau ^{+}p^{+}-hp^{-})i}\int_{B(x,2^{i-v+1})}%
\Big(\sum_{j=(v-i-1)^{+}}^{\infty }\Big(\frac{2^{j\alpha (y)q(y)}|\varphi
_{j}\ast f(y)|^{q(y)}}{\left\vert B(x,2^{i-v+1})\right\vert ^{q(y)\tau (y)}}%
\Big)\Big)^{p(y)/q(y)}dy
\end{eqnarray*}%
and since $\left\Vert f\right\Vert _{\mathfrak{F}_{p(\cdot ),q(\cdot
)}^{\alpha (\cdot ),\tau (\cdot )}}=1$, the lat term is bounded by%
\begin{equation*}
C\sum_{i=0}^{\infty }2^{(n\tau ^{+}p^{+}-hp^{-})i}<\infty
\end{equation*}%
for any $h>\tau ^{+}p^{+}/p^{-}$. The proof is completed by the scaling
argument.
\end{proof}

\begin{remark}
\label{new-est}$\mathrm{From}$ $\mathrm{this}$ $\mathrm{theorem}$ $\mathrm{we%
}$ $\mathrm{obtain}$%
\begin{equation}
2^{v(\alpha (x)+n(\tau (x)-1/p(x))}|\varphi _{v}\ast f(x)|\leq c\left\Vert
f\right\Vert _{\mathfrak{F}_{p(\cdot ),q(\cdot )}^{\alpha (\cdot ),\tau
(\cdot )}}  \label{emd}
\end{equation}%
$\mathrm{for}$ $\mathrm{any}$ $f\in \mathfrak{F}_{p(\cdot ),q(\cdot
)}^{\alpha (\cdot ),\tau (\cdot )}$, $x\in \mathbb{R}^{n}$, $\alpha ,\tau
\in C_{\mathrm{loc}}^{\log }$, $\mathbb{\tau }^{-}\geq 0$ $\mathrm{and}$ $%
p,q\in \mathcal{P}_{0}^{\log }$.
\end{remark}

In the following theorem we prove the possibility to define these spaces by
replacing $v\geq v_{P}^{+}$ by $v\geq 0$, in Definition \ref{B-F-def}. For
fixed exponents, see \cite{Si}.

\begin{theorem}
Let $\alpha ,\mathbb{\tau }\in C_{\mathrm{loc}}^{\log }$, $\mathbb{\tau }%
^{-}\geq 0$ and $p,q\in \mathcal{P}_{0}^{\log }$ with $0<q^{+},p^{+}<\infty $%
. If $\left( \mathbb{\tau }-\frac{1}{p}\right) ^{+}<0$ or\ $\left( \mathbb{%
\tau }-\frac{1}{p}\right) ^{+}\leq 0$\ and $q:=\infty $, then%
\begin{equation*}
\left\Vert f\right\Vert _{\mathfrak{F}_{p(\cdot ),q(\cdot )}^{\alpha (\cdot
),\tau (\cdot )}}^{\ast }=\sup_{P\in \mathcal{Q}}\Big\|\Big(\frac{2^{\alpha
\left( \cdot \right) }\varphi _{v}\ast f}{|P|^{\tau (\cdot )}}\chi _{P}\Big)%
_{v\geq 0}\Big\|_{L^{p(\cdot )}(\ell ^{q(\cdot )})},
\end{equation*}%
is an equivalent quasi-norm in $\mathfrak{F}_{p(\cdot ),q(\cdot )}^{\alpha
(\cdot ),\tau (\cdot )}$.
\end{theorem}

\begin{proof}
Clearly, it suffices to prove that $\left\Vert f\right\Vert _{\mathfrak{F}%
_{p(\cdot ),q(\cdot )}^{\alpha (\cdot ),\tau (\cdot )}}^{\ast }\lesssim
\left\Vert f\right\Vert _{\mathfrak{F}_{p(\cdot ),q(\cdot )}^{\alpha (\cdot
),\tau (\cdot )}}$. In view of the proof of previous theorem, we have%
\begin{equation*}
2^{v(\alpha (x)+n(\tau (x)-1/p(x))}|\varphi _{v}\ast f(x)|\leq c\left\Vert
f\right\Vert _{\mathfrak{F}_{p(\cdot ),q(\cdot )}^{\alpha (\cdot ),\tau
(\cdot )}}
\end{equation*}%
for any $x\in \mathbb{R}^{n}$. Then 
\begin{eqnarray*}
&&\Big\|\Big(\dsum\limits_{v=0}^{v_{P}}\left\vert \frac{2^{v\alpha (\cdot
)}\varphi _{v}\ast f}{|P|^{\tau (\cdot )}}\right\vert ^{q(\cdot )}\Big)%
^{1/q(\cdot )}\chi _{P}\Big\|_{p(\cdot )} \\
&\lesssim &\left\Vert f\right\Vert _{\mathfrak{F}_{p(\cdot ),q(\cdot
)}^{\alpha (\cdot ),\tau (\cdot )}}\Big\|\Big(\dsum\limits_{v=0}^{v_{P}}%
\left\vert 2^{(v_{P}-v)(\tau (\cdot )-1/p(\cdot ))+nv_{P}/p(\cdot
)}\right\vert ^{q(\cdot )}\Big)^{1/q(\cdot )}\chi _{P}\Big\|_{p(\cdot )} \\
&\lesssim &\left\Vert f\right\Vert _{\mathfrak{F}_{p(\cdot ),q(\cdot
)}^{\alpha (\cdot ),\tau (\cdot )}}\left\Vert \left\vert 2^{nv_{P}/p(\cdot
)}\right\vert ^{q(\cdot )}\chi _{P}\right\Vert _{\frac{p(\cdot )}{q(\cdot )}%
}\lesssim 1,
\end{eqnarray*}%
since $\left( \mathbb{\tau }-\frac{1}{p}\right) ^{+}<0$. The case $q:=\infty 
$ can be easily solved.
\end{proof}

Let $\Phi $ and $\varphi $ satisfy, respectively $\mathrm{\eqref{Ass1}}$ and 
$\mathrm{\eqref{Ass2}}$. By \cite[pp. 130--131]{FJ90}, there exist \
functions $\Psi \in \mathcal{S}(\mathbb{R}^{n})$ satisfying $\mathrm{%
\eqref{Ass1}}$ and $\psi \in \mathcal{S}(\mathbb{R}^{n})$ satisfying $%
\mathrm{\eqref{Ass2}}$ such that for all $\xi \in \mathbb{R}^{n}$%
\begin{equation}
\mathcal{F}\widetilde{\Phi }(\xi )\mathcal{F}\Psi (\xi )+\sum_{j=1}^{\infty }%
\mathcal{F}\widetilde{\varphi }(2^{-j}\xi )\mathcal{F}\psi (2^{-j}\xi
)=1,\quad \xi \in \mathbb{R}^{n}.  \label{Ass4}
\end{equation}

Furthermore, we have the following identity for all $f\in \mathcal{S}%
^{\prime }(\mathbb{R}^{n})$; see \cite[(12.4)]{FJ90}%
\begin{eqnarray*}
f &=&\Psi \ast \widetilde{\Phi }\ast f+\sum_{v=1}^{\infty }\psi _{v}\ast 
\widetilde{\varphi }_{v}\ast f \\
&=&\sum_{m\in \mathbb{Z}^{n}}\widetilde{\Phi }\ast f(m)\Psi (\cdot
-m)+\sum_{v=1}^{\infty }2^{-vn}\sum_{m\in \mathbb{Z}^{n}}\widetilde{\varphi }%
_{v}\ast f(2^{-v}m)\psi _{v}(\cdot -2^{-v}m).
\end{eqnarray*}%
Recall that the $\varphi $-transform $S_{\varphi }$ is defined by setting $%
(S_{\varphi })_{0,m}=\langle f,\Phi _{m}\rangle $ where $\Phi _{m}(x)=\Phi
(x-m)$ and $(S_{\varphi })_{v,m}=\langle f,\varphi _{v,m}\rangle $ where $%
\varphi _{v,m}(x)=2^{vn/2}\varphi (2^{v}x-m)$ and $v\in \mathbb{N}$. The
inverse $\varphi $-transform $T_{\psi }$ is defined by 
\begin{equation*}
T_{\psi }\lambda =\sum_{m\in \mathbb{Z}^{n}}\lambda _{0,m}\Psi
_{m}+\sum_{v=1}^{\infty }\sum_{m\in \mathbb{Z}^{n}}\lambda _{v,m}\psi _{v,m},
\end{equation*}%
where $\lambda =\{\lambda _{v,m}\in \mathbb{C}:v\in \mathbb{N}_{0},m\in 
\mathbb{Z}^{n}\}$, see \cite{FJ90}.

For any $\gamma \in \mathbb{Z}$, we put%
\begin{equation*}
\left\Vert f\right\Vert _{\mathfrak{F}_{p(\cdot ),q(\cdot )}^{\alpha (\cdot
),\tau (\cdot )}}^{\ast }:=\sup_{P\in \mathcal{Q}}\Big\|\Big(\frac{%
2^{v\alpha \left( \cdot \right) }\varphi _{v}\ast f}{|P|^{\tau (\cdot )}}%
\chi _{P}\Big)_{v\geq v_{P}^{+}-\gamma }\Big\|_{L^{p(\cdot )}(\ell ^{q(\cdot
)})}<\infty
\end{equation*}%
where $\varphi _{-\gamma }$ is replaced by $\Phi _{-\gamma }$.

\begin{lemma}
\label{new-equinorm3}Let $\alpha ,\mathbb{\tau }\in C_{\mathrm{loc}}^{\log }$%
, $\mathbb{\tau }^{-}>0$, $p,q\in \mathcal{P}_{0}^{\log }$ and $%
0<q^{+},p^{+}<\infty $. The quasi-norms $\left\Vert f\right\Vert _{\mathfrak{%
F}_{p(\cdot ),q(\cdot )}^{\alpha (\cdot ),\tau (\cdot )}}^{\ast }$ and $%
\left\Vert f\right\Vert _{\mathfrak{F}_{p(\cdot ),q(\cdot )}^{\alpha (\cdot
),\tau (\cdot )}}$ are equivalent with equivalent constants depending on $%
\gamma $.
\end{lemma}

\begin{proof}
{The proof follows the ideas in} \cite{WYY}\ and \cite{D6}. By similarity,
we only consider the case $\gamma >0$. First let us prove that $\left\Vert
f\right\Vert _{\mathfrak{F}_{p(\cdot ),q(\cdot )}^{\alpha (\cdot ),\tau
(\cdot )}}^{\ast }\leq c\left\Vert f\right\Vert _{\mathfrak{F}_{p(\cdot
),q(\cdot )}^{\alpha (\cdot ),\tau (\cdot )}}$. By the scaling argument, it
suffices to consider the case $\left\Vert f\right\Vert _{\mathfrak{F}%
_{p(\cdot ),q(\cdot )}^{\alpha (\cdot ),\tau (\cdot )}}=1$ and show that the
modular of $f$ on the left-hand side is bounded. In particular, we will show
that 
\begin{equation*}
\Big\|\Big(\sum_{v=v_{P}^{+}-\gamma }^{\infty }\left\vert \frac{2^{v\alpha
(\cdot )}\varphi _{v}\ast f}{|P|^{\tau (\cdot )}}\right\vert ^{q(\cdot )}%
\Big)^{1/q(\cdot )}\chi _{P}\Big\|_{p(\cdot )}\leq c
\end{equation*}%
for any dyadic cube $P$. We will use the same arguments of \cite{D6}. As in 
\cite[Lemma 2.6]{WYY}, it suffices to prove that for all dyadic cube $P$
with $l(P)\geq 1$,%
\begin{equation*}
I_{P}=\Big\|\Big(\sum_{v=-\gamma }^{0}\left\vert \frac{2^{v\alpha (\cdot
)}\varphi _{v}\ast f}{|P|^{\tau (\cdot )}}\right\vert ^{q(\cdot )}\Big)%
^{1/q(\cdot )}\chi _{P}\Big\|_{p(\cdot )}\leq c
\end{equation*}%
and for all dyadic cube $P$ with $l(P)<1$,%
\begin{equation*}
J_{P}=\Big\|\Big(\sum_{v=v_{P}-\gamma }^{v_{P}-1}\left\vert \frac{2^{v\alpha
(\cdot )}\varphi _{v}\ast f}{|P|^{\tau (\cdot )}}\right\vert ^{q(\cdot )}%
\Big)^{1/q(\cdot )}\chi _{P}\Big\|_{p(\cdot )}\leq c.
\end{equation*}%
The estimate of $I_{P}$, clearly follows from the inequality $\left\Vert 
\frac{\varphi _{v}\ast f}{|P|^{\tau (\cdot )}}\chi _{P}\right\Vert _{p(\cdot
)}\leq c$ for any $v=-\gamma ,...,0$ and any dyadic cube $P$ with $l(P)\geq
1 $. By $\mathrm{\eqref{Ass1}}$ and $\mathrm{\eqref{Ass2}}$, there exist $%
\omega _{v}\in \mathcal{S}(\mathbb{R}^{n})$, $v=-\gamma ,\cdot \cdot \cdot
,-1$ and $\eta _{1},\eta _{2}\in \mathcal{S}(\mathbb{R}^{n})$ such that%
\begin{equation*}
\varphi _{v}=\omega _{v}\ast \Phi ,\quad v=-\gamma ,\cdot \cdot \cdot ,-1%
\text{\quad and\quad }\varphi =\varphi _{0}=\eta _{1}\ast \Phi +\eta
_{2}\ast \varphi _{1}.
\end{equation*}%
Hence $\varphi _{v}\ast f=\omega _{v}\ast \Phi \ast f$ for $v=-\gamma
,...,-1 $ and $\varphi _{0}\ast f=\eta _{1}\ast \Phi \ast f+\eta _{2}\ast
\varphi _{1}\ast f$. Applying Lemma \ref{key-estimate1}, $\mathrm{%
\eqref{mod-est}}$ and the fact that $\left\Vert f\right\Vert _{\mathfrak{F}%
_{p(\cdot ),q(\cdot )}^{\alpha (\cdot ),\tau (\cdot )}}\leq 1$ to estimate $%
\left\Vert \frac{\varphi _{v}\ast f}{|P|^{\tau (\cdot )}}\chi
_{P}\right\Vert _{p(\cdot )}$ by 
\begin{equation*}
C\left\Vert \Phi \ast f\right\Vert _{\widetilde{L_{\tau (\cdot )}^{p(\cdot )}%
}}+C\left\Vert \varphi _{1}\ast f\right\Vert _{\widetilde{L_{\tau (\cdot
)}^{p(\cdot )}}}\leq c.
\end{equation*}%
To estimate $J_{P}$, denote by $P(\gamma )$ the dyadic cube containing $P$
with $l(P(\gamma ))=2^{\gamma }l(P)$. If $v_{P}\geq \gamma +1$, applying the
fact that $v_{P(\gamma )}=v_{P}-\gamma $ and $P\subset P(\gamma )$, we then
have%
\begin{equation*}
J_{P}\lesssim \Big\|\Big(\sum_{v=v_{P(\gamma )}}^{v_{P}-1}\left\vert \frac{%
2^{v\alpha (\cdot )}\varphi _{v}\ast f}{|P(\gamma )|^{\tau (\cdot )}}%
\right\vert ^{q(\cdot )}\Big)^{1/q(\cdot )}\chi _{P(\gamma )}\Big\|_{p(\cdot
)}\leq c.
\end{equation*}%
If $1\leq v_{P}\leq \gamma $, we write $J_{P}=\sum_{v=v_{P}-\gamma
}^{-1}\cdot \cdot \cdot +\sum_{v=0}^{v_{P}-1}\cdot \cdot \cdot
=J_{P}^{1}+J_{P}^{2}$. Let $P(2^{v_{P}})$ the dyadic cube containing $P$
with $l(P(2^{v_{P}}))=2^{v_{P}}l(P)=1$. By the fact that $\frac{%
|P(2^{v_{P}})|^{\tau (\cdot )}}{|P|^{\tau (\cdot )}}\lesssim 2^{nv_{P}\tau
^{+}}\lesssim c(\gamma )$ we have%
\begin{equation*}
J_{P}^{2}\lesssim \Big\|\Big(\sum_{v=v_{P(2^{v_{P}})}}^{v_{P}-1}\left\vert 
\frac{2^{v\alpha (\cdot )}\varphi _{v}\ast f}{|P(2^{v_{P}})|^{\tau (\cdot )}}%
\right\vert ^{q(\cdot )}\Big)^{1/q(\cdot )}\chi _{P(2^{v_{P}})}\Big\|%
_{p(\cdot )}\leq c.
\end{equation*}%
By a similar argument to the estimate for $I_{P}$, we see that $%
J_{P}^{1}\leq c$. For the converse estimate, it suffices to show that 
\begin{equation*}
\Big\|\Big|\frac{\Phi \ast f}{|P|^{\tau (\cdot )}}\Big|^{q(\cdot )}\chi _{P}%
\Big\|_{p(\cdot )/q(\cdot )}\leq c
\end{equation*}%
for all $P\in \mathcal{Q}$ with $l(P)\geq 1$ and all $f\in \mathfrak{F}%
_{p(\cdot ),q(\cdot )}^{\alpha (\cdot ),\tau (\cdot )}$ with $\left\Vert
f\right\Vert _{\mathfrak{F}_{p(\cdot ),q(\cdot )}^{\alpha (\cdot ),\tau
(\cdot )}}^{\ast }\leq 1$. This claim can be reformulated as showing that $%
\left\Vert \frac{\Phi \ast f}{|P|^{\tau (\cdot )}}\chi _{P}\right\Vert
_{p(\cdot )}\leq c$. Using the fact that there exist $\varrho _{v}\in 
\mathcal{S}(\mathbb{R}^{n})$, $v=-\gamma ,$\textperiodcentered
\textperiodcentered \textperiodcentered $,1$, such that $\Phi \ast f=\varrho
_{-\gamma }\ast \Phi _{-\gamma }\ast f+\sum_{v=1-\gamma }^{1}\varrho
_{v}\ast \varphi _{v}\ast f$, see \cite[p. 130]{FJ90}. Applying Lemma \ref%
{key-estimate1} we obtain%
\begin{equation*}
\left\Vert \varrho _{-\gamma }\ast \Phi _{-\gamma }\ast f\right\Vert _{%
\widetilde{L_{\tau (\cdot )}^{p(\cdot )}}}\lesssim \left\Vert \Phi _{-\gamma
}\ast f\right\Vert _{\widetilde{L_{\tau (\cdot )}^{p(\cdot )}}}\leq c,
\end{equation*}%
and%
\begin{equation*}
\left\Vert \varrho _{v}\ast \varphi _{v}\ast f\right\Vert _{\widetilde{%
L_{\tau (\cdot )}^{p(\cdot )}}}\lesssim \left\Vert \varphi _{v}\ast
f\right\Vert _{\widetilde{L_{\tau (\cdot )}^{p(\cdot )}}}\leq c,\quad
v=1-\gamma ,\cdot \cdot \cdot ,1,
\end{equation*}%
by using $\mathrm{\eqref{mod-est}}$ and the fact that $\left\Vert
f\right\Vert _{\mathfrak{F}_{p(\cdot ),q(\cdot )}^{\alpha (\cdot ),\tau
(\cdot )}}^{\ast }\leq 1$. The proof is complete.
\end{proof}

\begin{definition}
\label{sequence-space}Let $p,q\in \mathcal{P}_{0}$, $\tau :\mathbb{R}%
^{n}\rightarrow \mathbb{R}^{+}$\ and let $\alpha $ $:\mathbb{R}%
^{n}\rightarrow \mathbb{R}$. Then for all complex valued sequences $\lambda
=\{\lambda _{v,m}\in \mathbb{C}\}_{v\in \mathbb{N}_{0},m\in \mathbb{Z}^{n}}$
we define%
\begin{equation*}
\mathfrak{f}_{p(\cdot ),q(\cdot )}^{\alpha (\cdot ),\tau (\cdot )}:=\Big\{%
\lambda :\left\Vert \lambda \right\Vert _{\mathfrak{f}_{p(\cdot ),q(\cdot
)}^{\alpha (\cdot ),\tau (\cdot )}}<\infty \Big\},
\end{equation*}%
where%
\begin{equation*}
\left\Vert \lambda \right\Vert _{\mathfrak{f}_{p(\cdot ),q(\cdot )}^{\alpha
(\cdot ),\tau (\cdot )}}:=\sup_{P\in \mathcal{Q}}\Big\|\Big(\frac{%
\sum\limits_{m\in \mathbb{Z}^{n}}2^{v(\alpha \left( \cdot \right)
+n/2)}\lambda _{v,m}\chi _{v,m}}{|P|^{\tau (\cdot )}}\chi _{P}\Big)_{v\geq
v_{P}^{+}}\Big\|_{L^{p(\cdot )}(\ell ^{q(\cdot )})}.
\end{equation*}
\end{definition}

If we replace dyadic cubes $P$ by arbitrary balls $B_{J}$ of $\mathbb{R}^{n}$
with $J\in \mathbb{Z}$, we then obtain equivalent quasi-norms, where the
supremum is taken over all $J\in \mathbb{Z}$\ and all balls $B_{J}$\ of $%
\mathbb{R}^{n}$. As in \cite{D6}, we obtain the following two proprties.

\begin{lemma}
Let $\alpha ,\mathbb{\tau }\in C_{\mathrm{loc}}^{\log }$, $\tau ^{-}\geq 0$, 
$p,q\in \mathcal{P}_{0}^{\log }$, $0<q^{+},p^{+}<\infty $, $v\in \mathbb{N}%
_{0},m\in \mathbb{Z}^{n}$, $x\in Q_{v,m}$ and $\lambda \in \mathfrak{f}%
_{p(\cdot ),q(\cdot )}^{\alpha (\cdot ),\tau (\cdot )}$. Then there exists $%
c>0$ independent of $v$ and $m$ such that%
\begin{equation*}
|\lambda _{v,m}|\leq c\text{ }2^{-v(\alpha (x)+n/2)}|Q_{v,m}|^{\tau
(x)}\left\Vert \lambda \right\Vert _{\mathfrak{f}_{p(\cdot ),q(\cdot
)}^{\alpha (\cdot ),\tau (\cdot )}}\left\Vert \chi _{v,m}\right\Vert
_{p(\cdot )}^{-1}.
\end{equation*}
\end{lemma}

\begin{lemma}
Let $\alpha ,\mathbb{\tau }\in C_{\mathrm{loc}}^{\log }$, $\mathbb{\tau }%
^{-}\geq 0$, $p,q\in \mathcal{P}_{0}^{\log }$\ with $0<q^{+},p^{+}<\infty $
and $\Psi $, $\psi \in \mathcal{S}(\mathbb{R}^{n})$ satisfy, respectively, $%
\mathrm{\eqref{Ass1}}$ and $\mathrm{\eqref{Ass2}}$. Then for all $\lambda
\in \mathfrak{f}_{p(\cdot ),q(\cdot )}^{\alpha (\cdot ),\tau (\cdot )}$%
\begin{equation*}
T_{\psi }\lambda :=\sum_{m\in \mathbb{Z}^{n}}\lambda _{0,m}\Psi
_{m}+\sum_{v=1}^{\infty }\sum_{m\in \mathbb{Z}^{n}}\lambda _{v,m}\psi _{v,m},
\end{equation*}%
converges in $\mathcal{S}^{\prime }(\mathbb{R}^{n})$; moreover, $T_{\psi }:%
\mathfrak{f}_{p(\cdot ),q(\cdot )}^{\alpha (\cdot ),\tau (\cdot
)}\rightarrow \mathcal{S}^{\prime }(\mathbb{R}^{n})$ is continuous.
\end{lemma}

For a sequence $\lambda =\{\lambda _{v,m}\in \mathbb{C}\}_{v\in \mathbb{N}%
_{0},m\in \mathbb{Z}^{n}}$, $0<r\leq \infty $ and a fixed $d>0$, set%
\begin{equation*}
\lambda _{v,m,r,d}^{\ast }:=\Big(\sum_{h\in \mathbb{Z}^{n}}\frac{|\lambda
_{v,h}|^{r}}{(1+2^{v}|2^{-v}h-2^{-v}m|)^{d}}\Big)^{1/r}
\end{equation*}%
and $\lambda _{r,d}^{\ast }:=\{\lambda _{v,m,r,d}^{\ast }\in \mathbb{C}%
\}_{v\in \mathbb{N}_{0},m\in \mathbb{Z}^{n}}$.

\begin{lemma}
\label{lamda-equi}Let $\alpha ,\mathbb{\tau }\in C_{\mathrm{loc}}^{\log }$, $%
\mathbb{\tau }^{-}>0$, $p,q\in \mathcal{P}_{0}^{\log }$, $%
0<q^{+},p^{+}<\infty $, $w>n+c_{\log }(1/q)+c_{\log }\left( \tau \right) $
and $\frac{\tau ^{+}}{\tau ^{-}}<\frac{\min \left( p^{-},q^{-}\right) }{r}$.
Let $a=r\max (c_{\log }(\alpha ),\alpha ^{+}-\alpha ^{-})$. Then%
\begin{equation*}
\left\Vert \lambda _{r,d}^{\ast }\right\Vert _{\mathfrak{f}_{p(\cdot
),q(\cdot )}^{\alpha (\cdot ),\tau (\cdot )}}\approx \left\Vert \lambda
\right\Vert _{\mathfrak{f}_{p(\cdot ),q(\cdot )}^{\alpha (\cdot ),\tau
(\cdot )}}
\end{equation*}%
where 
\begin{equation*}
d>nr\tau ^{+}+n\tau ^{-}+n+a+w.
\end{equation*}
\end{lemma}

The proof of this lemma is postponed to the Appendix. By this result and by
the same arguments given in \cite[Theorem 3.14]{D6} we obtain the following
statement.

\begin{theorem}
\label{phi-tran}Let $\alpha ,\mathbb{\tau }\in C_{\mathrm{loc}}^{\log }$, $%
\mathbb{\tau }^{-}>0$, $p,q\in \mathcal{P}_{0}^{\log }$ and $%
0<q^{+},p^{+}<\infty $. \textit{Suppose that }$\Phi $, $\Psi \in \mathcal{S}(%
\mathbb{R}^{n})$ satisfying $\mathrm{\eqref{Ass1}}$ and $\varphi ,\psi \in 
\mathcal{S}(\mathbb{R}^{n})$ satisfy $\mathrm{\eqref{Ass2}}$ such that $%
\mathrm{\eqref{Ass4}}$ holds. The operators $S_{\varphi }:\mathfrak{F}%
_{p\left( \cdot \right) ,q\left( \cdot \right) }^{\alpha \left( \cdot
\right) ,\tau (\cdot )}\rightarrow \mathfrak{f}_{p(\cdot ),q(\cdot
)}^{\alpha (\cdot ),\tau (\cdot )}$ and $T_{\psi }:\mathfrak{f}_{p\left(
\cdot \right) ,q\left( \cdot \right) }^{\alpha \left( \cdot \right) ,\tau
(\cdot )}\rightarrow \mathfrak{F}_{p\left( \cdot \right) ,q\left( \cdot
\right) }^{\alpha \left( \cdot \right) ,\tau (\cdot )}$ are bounded.
Furthermore, $T_{\psi }\circ S_{\varphi }$ is the identity on $\mathfrak{F}%
_{p\left( \cdot \right) ,q\left( \cdot \right) }^{\alpha \left( \cdot
\right) ,\tau (\cdot )}$.
\end{theorem}

From \ this theorem, we obtain the next important property of spaces $%
\mathfrak{F}_{p\left( \cdot \right) ,q\left( \cdot \right) }^{\alpha \left(
\cdot \right) ,\tau (\cdot )}$.

\begin{corollary}
Let $\alpha ,\mathbb{\tau }\in C_{\mathrm{loc}}^{\log }$, $\tau ^{-}>0$, $%
p,q\in \mathcal{P}_{0}^{\log }$ and $0<p^{+},q^{+}<\infty $, The definition
of the spaces $\mathfrak{F}_{p\left( \cdot \right) ,q\left( \cdot \right)
}^{\alpha \left( \cdot \right) ,\tau (\cdot )}$ is independent of the
choices of $\Phi $ and $\varphi $.
\end{corollary}

\section{Embeddings}

For the spaces $\mathfrak{F}_{p(\cdot ),q(\cdot )}^{\alpha (\cdot ),\tau
(\cdot )}$ introduced above we want to show some embedding theorems. We say
a quasi-Banach space $A_{1}$ is continuously embedded in another
quasi-Banach space $A_{2}$, $A_{1}\hookrightarrow A_{2}$, if $A_{1}\subset
A_{2}$ and there is a $c>0$ such that $\left\Vert f\right\Vert _{A_{2}}\leq
c\left\Vert f\right\Vert _{A_{1}}$ for all $f\in A_{1}$. We begin with the
following elementary embeddings.

\begin{theorem}
\label{Elem-emb}Let $\alpha ,\mathbb{\tau }\in C_{\mathrm{loc}}^{\log }$, $%
\tau ^{-}>0$ and $p,q,q_{0},q_{1}\in \mathcal{P}_{0}^{\log }$ with $%
p^{+},q^{+},q_{0}^{+},q_{1}^{+}<\infty $.\newline
(i) If $q_{0}\leq q_{1}$, then 
\begin{equation*}
\mathfrak{F}_{p(\cdot ),q_{0}(\cdot )}^{\alpha (\cdot ),\tau (\cdot
)}\hookrightarrow \mathfrak{F}_{p(\cdot ),q_{1}(\cdot )}^{\alpha (\cdot
),\tau (\cdot )}.
\end{equation*}%
(ii) If $(\alpha _{0}-\alpha _{1})^{-}>0$, then 
\begin{equation*}
\mathfrak{F}_{{p(\cdot )},q_{0}{(\cdot )}}^{\alpha _{0}(\cdot ){,\tau (\cdot
)}}\hookrightarrow \mathfrak{F}_{{p(\cdot )},q_{1}{(\cdot )}}^{\alpha
_{1}(\cdot ){,\tau (\cdot )}}.
\end{equation*}
\end{theorem}

The proof can be obtained by using the properties of $L^{p(\cdot )}(\ell
^{q(\cdot )})$ spaces. We next consider embeddings of Sobolev-type. It is
well-known that%
\begin{equation*}
F_{{p}_{0},q}^{{\alpha }_{0}{,\tau }}\hookrightarrow F_{{p}_{1},q}^{{\alpha }%
_{1}{,\tau }}.
\end{equation*}%
if ${\alpha }_{0}-n/{p}_{0}={\alpha }_{1}-n/{p}_{1}$, where $0<{p}_{0}<{p}%
_{1}<\infty ,0\leq {\tau }<\infty $ and $0<q\leq \infty $ (see e.g. \cite[%
Corollary 2.2]{WYY}). In the following theorem we generalize these
embeddings to variable exponent case.

\begin{theorem}
Let $\alpha _{0},\alpha _{1},\mathbb{\tau }\in C_{\mathrm{loc}}^{\log }$, $%
\tau ^{-}>0$ and $p_{0},p_{1},q\in \mathcal{P}_{0}^{\log }$ with $%
0<p_{0}^{+},p_{1}^{+},q^{+}<\infty $. If ${\alpha }_{0}>{\alpha }_{1}$\ and $%
{\alpha }_{0}{(x)-}\frac{n}{p_{0}(x)}={\alpha }_{1}{(x)-}\frac{n}{p_{1}(x)}$%
\ with $0<\Big(\frac{{p}_{0}}{{p}_{1}}\Big)^{+}<1$, then 
\begin{equation*}
\mathfrak{F}_{{p}_{0}{(\cdot )},q{(\cdot )}}^{{\alpha }_{0}{(\cdot ),\tau
(\cdot )}}\hookrightarrow \mathfrak{F}_{{p}_{1}{(\cdot )},\infty }^{{\alpha }%
_{1}{(\cdot ),\tau (\cdot )}}.
\end{equation*}
\end{theorem}

\begin{proof}
We use some\ idea of\ \cite{JS08}. By homogeneity, it suffices to consider
the case $\left\Vert f\right\Vert _{\mathfrak{F}_{{p}_{0}{(\cdot )},\infty
}^{{\alpha }_{0}{(\cdot ),\tau (\cdot )}}}=1$. We will provet that%
\begin{equation*}
\Big\|\Big(\frac{2^{v\alpha _{1}\left( \cdot \right) }\varphi _{v}\ast f}{%
|P|^{\tau (\cdot )}}\chi _{P}\Big)_{v\geq v_{P}^{+}}\Big\|_{L^{p_{1}(\cdot
)}(\ell ^{q(\cdot )})}\lesssim 1
\end{equation*}%
for any dyadic cube $P$. We will study two cases in particular.

\textit{Case 1.} $|P|>1$. Let $Q_{v}\subset P$ be a cube, with $\ell \left(
Q_{v}\right) =2^{-v}$ and $x\in Q_{v}\subset P$. By Lemma \ref{r-trick} we
have for any $m>n$, $0<r<\frac{\tau ^{-}p_{0}^{-}}{\tau ^{+}}$ 
\begin{equation*}
|\varphi _{v}\ast f(x)|\leq c\left( \eta _{v,m}\ast |\varphi _{v}\ast
f|^{r}(x)\right) ^{1/r},
\end{equation*}%
where $c>0$ independent of $v$. We have%
\begin{eqnarray*}
&&\eta _{v,m}\ast |\varphi _{v}\ast f|^{r}(x) \\
&=&2^{vn}\int_{\mathbb{R}^{n}}\frac{\left\vert \varphi _{v}\ast
f(z)\right\vert ^{r}}{\left( 1+2^{v}\left\vert x-z\right\vert \right) ^{m}}dz
\\
&=&\int_{3Q_{v}}\cdot \cdot \cdot dz+\sum_{k=(k_{1},...,k_{n})\in \mathbb{Z}%
^{n},\max_{i=1,...,n}|k_{i}|\geq 2}\int_{Q_{v}+kl(Q_{v})}\cdot \cdot \cdot
dz.
\end{eqnarray*}%
Let $z\in Q_{v}+kl(Q_{v})$ with $k\in \mathbb{Z}^{n}$ and $|k|>4\sqrt{n}$.
Then $\left\vert x-z\right\vert \geq \left\vert k\right\vert 2^{-v-1}$ and
the second integral is boundd by 
\begin{equation*}
\left\vert k\right\vert ^{-m}2^{vn}\int_{Q_{v}+kl(Q_{v})}|\varphi _{v}\ast
f(z)|^{r}dz=\left\vert k\right\vert ^{-m}M_{Q_{v}+kl(Q_{v})}|\varphi
_{v}\ast f|^{r}(x).
\end{equation*}%
We put $\sigma =\frac{{p}_{0}}{{p}_{1}}\in ]0,1[$ and $s_{0}=\alpha
_{1}-n/p_{1}\in C_{\text{loc}}^{\log }$. Since $\sigma \alpha _{0}+\left(
1-\sigma \right) s_{0}=\alpha _{1}$, Lemma \ref{BM-lemma} gives for any $%
x\in \mathbb{R}^{n}$%
\begin{eqnarray*}
&&\Big\|\Big(\frac{2^{\alpha _{1}(x)v}\varphi _{v}\ast f(x)}{\left\vert
P\right\vert ^{\tau (x)}}\chi _{P}(x)\Big)_{v\geq 0}\Big\|_{\ell ^{q(x)}} \\
&\leq &\Big\|\Big(\frac{2^{\alpha _{0}(x)v}\varphi _{v}\ast f(x)}{\left\vert
P\right\vert ^{\tau (x)}}\chi _{P}(x)\Big)_{v\geq 0}\Big\|_{\ell ^{_{\infty
}}}^{\sigma (x)}\Big\|\Big(\frac{2^{s_{0}(x)v}\varphi _{v}\ast f(x)}{%
\left\vert P\right\vert ^{\tau (x)}}\chi _{P}(x)\Big)_{v\geq 0}\Big\|_{\ell
^{_{\infty }}}^{1-\sigma (x)}.
\end{eqnarray*}%
This expression in $L^{p_{1}(\cdot )}$-norm is bounded. Indeed, 
\begin{eqnarray*}
&&\Big\|\Big(\frac{2^{s_{0}(x)vr}\left\vert \varphi _{v}\ast f(x)\right\vert
^{r}}{\left\vert P\right\vert ^{\tau (x)r}}\chi _{P}(x)\Big)_{v\geq 0}\Big\|%
_{\ell ^{_{\infty }}} \\
&\lesssim &\sum_{k\in \mathbb{Z}^{n},|k|\leq 4\sqrt{n}}\Big\|\Big(\frac{%
2^{s_{0}(x)vr}g_{v,k,m,N}(x)}{\left\vert P\right\vert ^{\tau (x)r}}\chi
_{P}(x)\Big)_{v\geq 0}\Big\|_{\ell ^{_{\infty }}}+ \\
&&\sum_{k\in \mathbb{Z}^{n},|k|>4\sqrt{n}}\left\vert k\right\vert
^{n-m+n\tau ^{+}r+\sigma ^{+}}\Big\|\Big(\frac{%
2^{s_{0}(x)vr}|k|^{-n}g_{v,k,m,N}(x)}{\left\vert P\right\vert ^{\tau (x)r}}%
\chi _{P}(x)\Big)_{v\geq 0}\Big\|_{\ell ^{_{\infty }}},
\end{eqnarray*}%
where $\varrho (\cdot )=n-\frac{ndr}{p_{0}(\cdot )\tau \left( \cdot \right) }
$ and 
\begin{equation*}
g_{v,k,m,N}=\left\{ 
\begin{array}{ccc}
M_{3Q_{v}}|\varphi _{v}\ast f|^{r} & \text{if} & |k|=0 \\ 
M_{Q_{v}+kl(Q_{v})}|\varphi _{v}\ast f|^{r} & \text{if} & 0<|k|\leq 4\sqrt{n}
\\ 
|k|^{-n\tau r}M_{Q_{v}+kl(Q_{v})}|k|^{-\varrho (\cdot )}|\varphi _{v}\ast
f|^{r} & \text{if} & |k|>4\sqrt{n}.%
\end{array}%
\right.
\end{equation*}%
We will prove that%
\begin{equation}
\Big\|\Big(\frac{2^{s_{0}(x)vr}|k|^{-n}g_{v,k,m,N}(x)}{\left\vert
P\right\vert ^{\tau (x)r}}\chi _{P}(x)\Big)_{v\geq 0}\Big\|_{\ell ^{\infty
}}\lesssim 1  \label{key-estT-L}
\end{equation}%
for any $k\in \mathbb{Z}^{n}$ with $|k|>4\sqrt{n}$. We take $d>0$ such that $%
\tau ^{+}<d<\frac{\tau ^{-}p_{0}^{-}}{r}$. Observe that $Q_{v}+kl(Q_{v})%
\subset Q\left( x,|k|2^{1-v}\right) $, by Theorem \ref{DHHR-estimate}, 
\begin{eqnarray*}
&&\frac{|k|^{-ndr}}{\left\vert P\right\vert ^{dr}}\left( M_{Q\left(
x,|k|2^{1-v}\right) }|k|^{-\varrho (\cdot )}2^{vs_{0}r}|\varphi _{v}\ast
f|^{r}(x)\right) ^{d/\tau \left( x\right) } \\
&\leq &M_{Q\left( x,|k|2^{1-v}\right) }\Big(\frac{2^{v\frac{s_{0}dr}{\tau }%
}|\varphi _{v}\ast f|^{rd/\tau }}{|k|^{ndr}\left\vert P\right\vert ^{dr}}%
\Big)(x)+C.
\end{eqnarray*}%
By H\"{o}lder's inequality this expression is bounded by%
\begin{equation*}
\Big\|\frac{2^{v\left( s_{0}+\frac{n}{p_{0}}\right) dr/\tau }}{%
|k|^{ndr}\left\vert P\right\vert ^{dr}}|\varphi _{v}\ast f|^{rd/\tau }\chi
_{Q\left( \cdot ,|k|2^{1-v}\right) }\Big\|_{p_{0}(\cdot )\tau \left( \cdot
\right) /dr}+C.
\end{equation*}%
Observe that $Q\left( x,|k|2^{1-v}\right) \subset Q\left(
c_{P},|k|2^{1-v_{P}}\right) $ for any $x\in Q_{v}$. The last norm is bounded
if and only if%
\begin{equation*}
\int_{Q\left( c_{P},|k|2^{1-v_{P}}\right) }\frac{2^{v\left( s_{0}(y)+\frac{n%
}{p_{0}(y)}\right) p_{0}(y)}|\varphi _{v}\ast f\left( y\right) |^{p_{0}(y)}}{%
\left( |k|^{n}\left\vert P\right\vert \right) ^{p_{0}(y)\tau \left( y\right)
}}dy\lesssim 1,
\end{equation*}%
which is equivalent to 
\begin{equation*}
\Big\|\frac{2^{v\left( s_{0}(\cdot )+\frac{n}{p_{0}(\cdot )}\right)
p_{0}(\cdot )}\left\vert \varphi _{v}\ast f\right\vert \chi _{Q\left(
c_{P},|k|2^{1-v_{P}}\right) }}{|Q\left( c_{P},|k|2^{1-v_{P}}\right) |^{\tau
(\cdot )}}\Big\|_{p_{0}(\cdot )}\lesssim 1,
\end{equation*}%
due to the fact that $s_{0}(\cdot )+\frac{n}{p_{0}(\cdot )}=\alpha
_{0}(\cdot )$ and $\left\vert Q\left( c_{P},|k|2^{1-v_{P}}\right)
\right\vert \geq 1$. Therefore, the sum $\sum_{k\in \mathbb{Z}^{n},|k|>4%
\sqrt{n}}\cdot \cdot \cdot $ is bounded by taking $m$ large enough. Similar
arguments with necessary modifications can be used to prove $\mathrm{%
\eqref{key-estT-L}}$ for any $|k|\leq 4\sqrt{n}$. Therefore,%
\begin{eqnarray*}
&&\int_{P}\Big(\sum\limits_{v=0}^{\infty }2^{v{\alpha }_{1}(x)q(x)}\frac{%
\left\vert \varphi _{v}\ast f(x)\right\vert ^{q(x)}}{\left\vert P\right\vert
^{\tau (x)q(x)}}\Big)^{p_{1}(x)/q(x)}dx \\
&\leq &c\int_{P}\Big\|\Big(\frac{2^{\alpha _{0}(x)v}\varphi _{v}\ast f(x)}{%
\left\vert P\right\vert ^{\tau (x)}}\Big)_{v\geq 0}\Big\|_{\ell ^{_{\infty
}}}^{\sigma (x)p_{1}(x)}dx \\
&=&c\int_{P}\Big\|\Big(\frac{2^{\alpha _{0}(x)v}\varphi _{v}\ast f(x)}{%
\left\vert P\right\vert ^{\tau (x)}}\Big)_{v\geq 0}\Big\|_{\ell ^{_{\infty
}}}^{p_{0}(x)}dx\leq c.
\end{eqnarray*}%
\textit{Case 2.} $|P|\leq 1$. {Since }$\tau ${\ is log-H\"{o}lder
continuous, we have }%
\begin{equation*}
\left\vert P\right\vert ^{-\tau (x)}\leq c\text{ }\left\vert P\right\vert
^{-\tau (y)}(1+2^{v_{P}}\left\vert x-y\right\vert )^{c_{\log }\left( \tau
\right) }\leq c\text{ }\left\vert P\right\vert ^{-\tau
(y)}(1+2^{v}\left\vert x-y\right\vert )^{c_{\log }\left( \tau \right) }
\end{equation*}%
for any $x,y\in \mathbb{R}^{n}$ and any $v\geq v_{P}$. Therefore,%
\begin{equation*}
\frac{1}{\left\vert P\right\vert ^{\tau (\cdot )}}\eta _{v,m}\ast \left(
|\varphi _{v}\ast f|\chi _{Q}\right) \lesssim \eta _{v,m-c_{\log }\left(
\tau \right) }\ast \frac{|\varphi _{v}\ast f|\chi _{Q}}{\left\vert
P\right\vert ^{\tau (\cdot )}}
\end{equation*}%
for any dyadic cube where $Q$. The arguments here are quite similar to those
used in the case $|P|>1$, where we did not need to use Theorem \ref%
{DHHR-estimate}, which could be used only to move $\left\vert P\right\vert
^{\tau (\cdot )}$ inside the convolution and hence the proof is complete.
\end{proof}

Let $\alpha ,\mathbb{\tau }\in C_{\mathrm{loc}}^{\log },\mathbb{\tau }%
^{-}>0,p,q\in \mathcal{P}_{0}^{\log }$ with $0<p^{+},q^{+}<\infty $. From $%
\mathrm{\eqref{emd}}$, we obtain%
\begin{equation*}
\mathfrak{F}_{p(\cdot ),q(\cdot )}^{\alpha (\cdot ),\tau (\cdot
)}\hookrightarrow F_{p(\cdot ),\infty }^{\alpha (\cdot )+n\tau (\cdot )-%
\frac{n}{p(\cdot )}}\hookrightarrow \mathcal{S}^{\prime }(\mathbb{R}^{n}).
\end{equation*}%
Similar arguments of \cite[Proposition 2.3]{WYY} can be used to prove that%
\begin{equation*}
\mathcal{S}(\mathbb{R}^{n})\hookrightarrow \mathfrak{F}_{p(\cdot ),q(\cdot
)}^{\alpha (\cdot ),\tau (\cdot )}.
\end{equation*}%
Therefore, we obtain the following result.

\begin{theorem}
Let $\alpha ,\mathbb{\tau }\in C_{\mathrm{loc}}^{\log },\mathbb{\tau }^{-}>0$
and $p,q\in \mathcal{P}_{0}^{\log }$ with $0<p^{+},q^{+}<\infty $. Then 
\begin{equation*}
\mathcal{S}(\mathbb{R}^{n})\hookrightarrow \mathfrak{F}_{p(\cdot ),q(\cdot
)}^{\alpha (\cdot ),\tau (\cdot )}\hookrightarrow \mathcal{S}^{\prime }(%
\mathbb{R}^{n}).
\end{equation*}
\end{theorem}

Now we establish some further embedding of the spaces $\mathfrak{F}_{p(\cdot
),q(\cdot )}^{\alpha (\cdot ),\tau (\cdot )}$.

\begin{theorem}
\textit{Let }$\alpha ,\tau \in C_{\mathrm{loc}}^{\log },\mathbb{\tau }^{-}>0$
\textit{and }$p,q\in \mathcal{P}_{0}^{\log }$ with $0<p^{+},q^{+}<\infty $%
\textit{.} If $(p_{2}-p_{1})^{+}\leq 0$, then%
\begin{equation*}
F_{{p}_{2}{(\cdot )},q{(\cdot )}}^{{\alpha (\cdot )+n\tau (\cdot )+\frac{n}{{%
p}_{2}{(\cdot )}}-\frac{n}{{p}_{1}{(\cdot )}}}}\hookrightarrow \mathfrak{F}_{%
{p}_{1}{(\cdot )},q{(\cdot )}}^{{\alpha (\cdot ),\tau (\cdot )}}.
\end{equation*}
\end{theorem}

\begin{proof}
Using the Sobolev embeddings%
\begin{equation*}
F_{{p}_{2}{(\cdot )},q{(\cdot )}}^{{\alpha (\cdot )+n\tau (\cdot )+\frac{n}{{%
p}_{2}{(\cdot )}}-\frac{n}{{p}_{1}{(\cdot )}}}}\hookrightarrow F_{{p}_{1}{%
(\cdot )},q{(\cdot )}}^{{\alpha (\cdot )+n\tau (\cdot )}},
\end{equation*}%
see \cite{V} it is sufficient to prove that $F_{{p}_{1}{(\cdot )},q{(\cdot )}%
}^{{\alpha (\cdot )+n\tau (\cdot )}}\hookrightarrow \mathfrak{F}_{{p}_{1}{%
(\cdot )},q{(\cdot )}}^{{\alpha (\cdot ),\tau (\cdot )}}$. We have 
\begin{equation*}
\sup_{P\in \mathcal{Q},|P|>1}\Big\|\Big(\frac{2^{v\alpha \left( \cdot
\right) }\varphi _{v}\ast f}{\left\vert P\right\vert ^{\tau (\cdot )}}\chi
_{P}\Big)_{v\geq v_{P}^{+}}\Big\|_{L^{{p}_{1}{(\cdot )}}(\ell ^{q(\cdot
)})}\leq \left\Vert \left( 2^{v\alpha \left( \cdot \right) }\varphi _{v}\ast
f\right) _{v\geq 0}\right\Vert _{L^{{p}_{1}{(\cdot )}}(\ell ^{q(\cdot )})}.
\end{equation*}%
In view of the definition of $F_{{p}_{1}{(\cdot )},q{(\cdot )}}^{{\alpha
(\cdot )}}$ spaces the last expression is bounded by $\left\Vert
f\right\Vert _{F_{{p}_{1}{(\cdot )},q(\cdot )}^{\alpha (\cdot )}}\leq
\left\Vert f\right\Vert _{F_{{p}_{1}{(\cdot )},q(\cdot )}^{\alpha (\cdot ){%
+n\tau (\cdot )}}}$. Now we have the estimates 
\begin{eqnarray*}
&&\sup_{P\in \mathcal{Q},|P|\leq 1}\Big\|\Big(\frac{2^{v\alpha \left( \cdot
\right) }\varphi _{v}\ast f}{\left\vert P\right\vert ^{\tau (\cdot )}}\chi
_{P}\Big)_{v\geq v_{P}^{+}}\Big\|_{L^{{p}_{1}{(\cdot )}}(\ell ^{q(\cdot )})}
\\
&\leq &\sup_{P\in \mathcal{Q},|P|\leq 1}\Big\|\Big(2^{v(\alpha \left( \cdot
\right) {+n\tau (\cdot ))+n\tau (\cdot )(v_{P}-v)}}\varphi _{v}\ast f\Big)%
_{v\geq v_{P}}\Big\|_{L^{{p}_{1}{(\cdot )}}(\ell ^{q(\cdot )})} \\
&\leq &\sup_{P\in \mathcal{Q},|P|\leq 1}\left\Vert \left( 2^{v(\alpha \left(
\cdot \right) {+n\tau (\cdot ))}}\varphi _{v}\ast f\right) _{v\geq
0}\right\Vert _{L^{{p}_{1}{(\cdot )}}(\ell ^{q(\cdot )})}\leq \left\Vert
f\right\Vert _{F_{{p}_{1}{(\cdot )},q(\cdot )}^{\alpha (\cdot ){+n\tau
(\cdot )}}},
\end{eqnarray*}%
which completes the proof.
\end{proof}

Let $p,u\in \mathcal{P}_{0}$ be such that $0<p(\cdot )\leq u(\cdot )<\infty $%
. The variable Morrey space $M_{p(\cdot )}^{u(\cdot )}$ is defined to be the
set of all $p(\cdot )$-locally Lebesgue-integrable functions $f$ on $\mathbb{%
R}^{n}$ such that%
\begin{equation*}
\left\Vert f\right\Vert _{M_{p(\cdot )}^{u(\cdot )}}=\sup_{P\in \mathcal{Q}}%
\Big\|\frac{f}{\left\vert P\right\vert ^{\frac{1}{p(\cdot )}-\frac{1}{%
u(\cdot )}}}\chi _{P}\Big\|_{{p(\cdot )}}<\infty .
\end{equation*}%
One recognizes immediately that if $p$ and $u$ are constants, then we obtain
the usual Morrey space $M_{p}^{u}$.

Let $\alpha :\mathbb{R}^{n}\rightarrow \mathbb{R}$ and $p,q,u\in \mathcal{P}%
_{0}$\ be such that $0<p(\cdot )\leq u(\cdot )<\infty $. Let $\Phi $ and $%
\varphi $ satisfy $\mathrm{\eqref{Ass1}}$ and $\mathrm{\eqref{Ass2}}$,
respectively and we put $\varphi _{v}=2^{vn}\varphi (2^{v}\cdot )$. The
Triebel-Lizorkin-Morrey space $\mathcal{E}_{u(\cdot ),p(\cdot )}^{\alpha
(\cdot ),q(\cdot )}$\ is the collection of all $f\in \mathcal{S}^{\prime }(%
\mathbb{R}^{n})$\ such that 
\begin{equation*}
\left\Vert f\right\Vert _{\mathcal{E}_{u(\cdot ),p(\cdot )}^{\alpha (\cdot
),q(\cdot )}}:=\Big\|\Big(2^{v\alpha \left( \cdot \right) }\varphi _{v}\ast f%
\Big)_{v\geq 0}\Big\|_{M_{p(\cdot )}^{u(\cdot )}(\ell ^{q(\cdot )})}<\infty ,
\end{equation*}%
where $\varphi _{0}$ is replaced by $\Phi $.

\begin{theorem}
Let $\alpha \in C_{\mathrm{loc}}^{\log }$ and $p,q,u\in \mathcal{P}%
_{0}^{\log }$\ with $0<p^{-}<p(\cdot )\leq u(\cdot )<u^{+}<\infty $\ and\ $%
0<q^{-}\leq q^{+}<\infty $.\ Then%
\begin{equation*}
\mathfrak{F}_{{p(\cdot )},q{(\cdot )}}^{{\alpha (\cdot ),}\frac{1}{p(\cdot )}%
-\frac{1}{u(\cdot )}}=\mathcal{E}_{u(\cdot ),p(\cdot )}^{\alpha (\cdot
),q(\cdot )},
\end{equation*}%
with equivalent quasi-norms.
\end{theorem}

\begin{proof}
Obviously we need to prove that $\left\Vert f\right\Vert _{\mathcal{E}%
_{u(\cdot ),p(\cdot )}^{\alpha (\cdot ),q(\cdot )}}\lesssim 1$\ for any\ $%
f\in \mathfrak{F}_{{p(\cdot )},q{(\cdot )}}^{{\alpha (\cdot ),}\frac{1}{%
p(\cdot )}-\frac{1}{u(\cdot )}}$ with $\left\Vert f\right\Vert _{\mathfrak{F}%
_{{p(\cdot )},q{(\cdot )}}^{{\alpha (\cdot ),}\frac{1}{p(\cdot )}-\frac{1}{%
u(\cdot )}}}\leq 1$. We must show the inequality%
\begin{equation*}
\Big\|\Big(\dsum_{v=0}^{v_{P}^{+}}\frac{\left\vert 2^{v\alpha \left( \cdot
\right) }\varphi _{v}\ast f\right\vert ^{q(\cdot )}}{\left\vert P\right\vert
^{(\frac{1}{p(\cdot )}-\frac{1}{u(\cdot )})q(\cdot )}}\chi _{P}\Big)%
^{1/q(\cdot )}\Big\|_{p(\cdot )}\lesssim 1
\end{equation*}%
for any dyadic cube $P$. Applying the property $\mathrm{\eqref{emd}}$, we
obtain%
\begin{equation*}
2^{v\alpha \left( x\right) }|\varphi _{v}\ast f(x)|\lesssim 2^{\frac{vn}{u(x)%
}}
\end{equation*}%
for any $x\in \mathbb{R}^{n}$, with the implicit constant independent of $k$
and $x$. Hence the last quasi-norm is bounded by%
\begin{equation*}
\Big\|2^{\frac{nv_{P}^{+}}{p(\cdot )}}\Big(%
\dsum_{v=0}^{v_{P}^{+}}2^{(v-v_{P}^{+})\frac{nq(\cdot )}{u(\cdot )}}\chi _{P}%
\Big)^{1/q(\cdot )}\Big\|_{L^{p(\cdot )}}\lesssim \Big\|2^{\frac{nv_{P}^{+}}{%
p(\cdot )}}\chi _{P}\Big\|_{p(\cdot )}\lesssim 1.
\end{equation*}%
This finishes the proof.
\end{proof}

\section{Atomic decomposition}

The idea of atomic decompositions leads back to M. Frazier and B. Jawerth in
their series of papers \cite{FJ86}, \cite{FJ90}, see also \cite{T3}. The
main goal of this section is to prove an atomic decomposition result for $%
\mathfrak{F}_{p(\cdot ),q(\cdot )}^{\alpha (\cdot ),\tau (\cdot )}$.

We now present a fundamental characterization of spaces under consideration.

\begin{theorem}
\label{fun-char}Let $\mathbb{\tau }$, $\alpha \in C_{\mathrm{loc}}^{\log
},\tau ^{-}>0$ and $p,q\in \mathcal{P}_{0}^{\log }$\ with $0<p^{-}\leq
p^{+}<\infty $ and $0<q^{-}\leq q^{+}<\infty $. Let $m$ be as in Lemma \ref%
{DHR-lemma1}, $a>\frac{\tau ^{+}m}{\tau ^{-}p^{-}}$ and $\Phi $ and $\varphi 
$ satisfy $\mathrm{\eqref{Ass1}}$ and $\mathrm{\eqref{Ass2}}$, respectively.
Then%
\begin{equation*}
\left\Vert f\right\Vert _{\mathfrak{F}_{p(\cdot ),q(\cdot )}^{\alpha (\cdot
),\tau (\cdot )}}^{\blacktriangledown }:=\sup_{P\in \mathcal{Q}}\Big\|\Big(%
\frac{\varphi _{v}^{\ast ,a}2^{v\alpha (\cdot )}f}{\left\vert P\right\vert
^{\tau (\cdot )}}\chi _{P}\Big)_{v\geq v_{P}^{+}}\Big\|_{L^{p(\cdot )}(\ell
^{q(\cdot )})}
\end{equation*}%
\textit{is an equivalent quasi-norm in }$\mathfrak{F}_{p(\cdot ),q(\cdot
)}^{\alpha (\cdot ),\tau (\cdot )}$.
\end{theorem}

\begin{proof}
It is easy to see that for any $f\in \mathcal{S}^{\prime }(\mathbb{R}^{n})$
with $\left\Vert f\right\Vert _{\mathfrak{F}_{p(\cdot ),q(\cdot )}^{\alpha
(\cdot ),\tau (\cdot )}}^{\blacktriangledown }<\infty $ and any $x\in 
\mathbb{R}^{n}$ we have%
\begin{equation*}
2^{v\alpha (x)}\left\vert \varphi _{v}\ast f(x)\right\vert \leq \varphi
_{v}^{\ast ,a}2^{v\alpha (\cdot )}f(x)\text{.}
\end{equation*}%
This shows that $\left\Vert f\right\Vert _{\mathfrak{F}_{p(\cdot ),q(\cdot
)}^{\alpha (\cdot ),\tau (\cdot )}}\leq \left\Vert f\right\Vert _{\mathfrak{F%
}_{p(\cdot ),q(\cdot )}^{\alpha (\cdot ),\tau (\cdot )}}^{\blacktriangledown
}$. Let $t>0$ be such that $a>\frac{m}{t}>\frac{m}{p^{-}}$. By Lemmas \ref%
{r-trick}\ and \ref{DHR-lemma} the estimates%
\begin{eqnarray}
2^{v\alpha \left( y\right) }\left\vert \varphi _{v}\ast f(y)\right\vert
&\leq &C_{1}\text{ }2^{v\alpha \left( y\right) }\left( \eta _{v,w}\ast
|\varphi _{v}\ast f|^{t}(y)\right) ^{1/t}  \notag \\
&\leq &C_{2}\text{ }\left( \eta _{v,w-c_{\log }(\alpha )}\ast (2^{v\alpha
\left( \cdot \right) }|\varphi _{v}\ast f|)^{t}(y)\right) ^{1/t}
\label{esti-conv}
\end{eqnarray}%
are true for any $y\in \mathbb{R}^{n},v\in \mathbb{N}_{0}$ and any $w>0$.
Now divide both sides of $\mathrm{\eqref{esti-conv}}$ by $\left(
1+2^{v}\left\vert x-y\right\vert \right) ^{a}$, in the right-hand side we
use the inequality%
\begin{equation*}
\left( 1+2^{v}\left\vert x-y\right\vert \right) ^{-a}\leq \left(
1+2^{v}\left\vert x-z\right\vert \right) ^{-a}\left( 1+2^{v}\left\vert
y-z\right\vert \right) ^{a},\quad x,y,z\in \mathbb{R}^{n},
\end{equation*}%
in the left-hand side take the supremum over $y\in \mathbb{R}^{n}$ and get
for all $f\in \mathfrak{F}_{p(\cdot ),q(\cdot )}^{\alpha (\cdot ),\tau
(\cdot )}$, any $x\in P$ any $v\geq v_{P}^{+}$ and any $w>at+c_{\log
}(\alpha )$%
\begin{equation*}
\left( \varphi _{v}^{\ast ,a}2^{v\alpha \left( \cdot \right) }f(x)\right)
^{t}\leq C_{2}\text{ }\eta _{v,at}\ast (2^{v\alpha \left( \cdot \right)
}|\varphi _{v}\ast f|)^{t}(x)
\end{equation*}%
where $C_{2}>0$ is independent of $x,v$ and $f$. Lemma \ref{DHR-lemma1}
gives that 
\begin{eqnarray*}
&&\left\Vert f\right\Vert _{\mathfrak{F}_{p(\cdot ),q(\cdot )}^{\alpha
(\cdot ),\tau (\cdot )}}^{\blacktriangledown }\lesssim C\sup_{P\in \mathcal{Q%
}}\Big\|\Big(\frac{\eta _{v,at}\ast (2^{v\alpha \left( \cdot \right)
}|\varphi _{v}\ast f|)^{t}}{\left\vert P\right\vert ^{\tau (\cdot )t}}\chi
_{P}\Big)_{v\geq v_{P}^{+}}\Big\|_{L^{\frac{p(\cdot )}{t}}(\ell ^{\frac{%
q(\cdot )}{t}})}^{1/t} \\
&\leq &C\left\Vert \left( 2^{v\alpha \left( \cdot \right) }\varphi _{v}\ast
f\right) _{v}\right\Vert _{L^{p(\cdot )}(\ell ^{q(\cdot )})}=C\left\Vert
f\right\Vert _{\mathfrak{F}_{p(\cdot ),q(\cdot )}^{\alpha (\cdot ),\tau
(\cdot )}}.
\end{eqnarray*}%
\noindent The proof is complete.
\end{proof}

Atoms are the building blocks for the atomic decomposition.

\begin{definition}
\label{Atom-Def}Let $K\in \mathbb{N}_{0},L+1\in \mathbb{N}_{0}$ and let $%
\gamma >1$. A $K$-times continuous differentiable function $a\in C^{K}(%
\mathbb{R}^{n})$ is called $[K,L]$-atom centered at $Q_{v,m}$, $v\in \mathbb{%
N}_{0}$ and $m\in \mathbb{Z}^{n}$, if

\begin{equation}
\mathrm{supp}\text{ }a\subseteq \gamma Q_{v,m}  \label{supp-cond}
\end{equation}

\begin{equation}
|\partial ^{\beta }a(x)|\leq 2^{v(|\beta |+1/2)}\text{,\quad for\quad }0\leq
|\beta |\leq K,x\in \mathbb{R}^{n}  \label{diff-cond}
\end{equation}%
and if%
\begin{equation}
\int_{\mathbb{R}^{n}}x^{\beta }a(x)dx=0,\text{\quad for\quad }0\leq |\beta
|\leq L\text{ and }v\geq 1.  \label{mom-cond}
\end{equation}
\end{definition}

If the atom $a$ located at $Q_{v,m}$, that means if it fulfills $\mathrm{%
\eqref{supp-cond}}$, then we will denote it by $a_{v,m}$. For $v=0$ or $L=-1$
there are no moment\ conditions\ $\mathrm{\eqref{mom-cond}}$ required.\vskip%
5pt

For proving the decomposition by atoms we need the following lemma, see
Frazier \& Jawerth \cite[Lemma 3.3]{FJ86}.

\begin{lemma}
\label{FJ-lemma}Let $\Phi $ and $\varphi $ satisfy, respectively, $\mathrm{%
\eqref{Ass1}}$ and $\mathrm{\eqref{Ass2}}$ and let $\varrho _{v,m}$ be an $%
\left[ K,L\right] $-atom. Then%
\begin{equation*}
\left\vert \varphi _{j}\ast \varrho _{v,m}(x)\right\vert \leq c\text{ }%
2^{(v-j)K+vn/2}\left( 1+2^{v}\left\vert x-x_{Q_{v,m}}\right\vert \right)
^{-M}
\end{equation*}%
if $v\leq j$, and%
\begin{equation*}
\left\vert \varphi _{j}\ast \varrho _{v,m}(x)\right\vert \leq c\text{ }%
2^{(j-v)(L+n+1)+vn/2}\left( 1+2^{j}\left\vert x-x_{Q_{v,m}}\right\vert
\right) ^{-M}
\end{equation*}%
if $v\geq j$, where $M$ is sufficiently large, $\varphi _{j}=2^{jn}\varphi
(2^{j}\cdot )$ and $\varphi _{0}$ is replaced by $\Phi $.
\end{lemma}

Now we come to the atomic decomposition theorem.

\begin{theorem}
\label{atomic-dec}\textit{Let }$\alpha ,\tau \in C_{\mathrm{loc}}^{\log
},\tau ^{-}>0$ \textit{and }$p,q\in \mathcal{P}_{0}^{\log }$\ with $%
0<p^{-}\leq p^{+}<\infty $ and $0<q^{-}\leq q^{+}<\infty $\textit{\ }Let $%
K,L+1\in \mathbb{N}_{0}$ such that%
\begin{equation}
K\geq ([\alpha ^{+}+n\tau ^{+}]+1)^{+},  \label{K,L,B-F-cond}
\end{equation}%
and%
\begin{equation}
L\geq \max (-1,[n(\frac{\tau ^{+}}{\tau ^{-}\min (1,p^{-},q^{-})}-1)-\alpha
^{-}]).  \label{K,L,B-cond}
\end{equation}%
Then $f\in \mathcal{S}^{\prime }(\mathbb{R}^{n})$ belongs to $\mathfrak{F}%
_{p(\cdot ),q(\cdot )}^{\alpha (\cdot ),\tau (\cdot )}$, if and only if it
can be represented as%
\begin{equation}
f=\sum\limits_{v=0}^{\infty }\sum\limits_{m\in \mathbb{Z}^{n}}\lambda
_{v,m}\varrho _{v,m},\text{ \ \ \ \ converging in }\mathcal{S}^{\prime }(%
\mathbb{R}^{n})\text{,}  \label{new-rep}
\end{equation}%
where $\varrho _{v,m}$ are $\left[ K,L\right] $-atoms and $\lambda
=\{\lambda _{v,m}\in \mathbb{C}:v\in \mathbb{N}_{0},m\in \mathbb{Z}^{n}\}\in 
\mathfrak{f}_{p(\cdot ),q(\cdot )}^{\alpha (\cdot ),\tau (\cdot )}$.
Furthermore, $\mathrm{inf}\left\Vert \lambda \right\Vert _{\mathfrak{f}%
_{p(\cdot ),q(\cdot )}^{\alpha (\cdot ),\tau (\cdot )}}$, where the infimum
is taken over admissible representations\ $\mathrm{\eqref{new-rep}}$\textrm{%
, }is an equivalent quasi-norm in $\mathfrak{F}_{p(\cdot ),q(\cdot
)}^{\alpha (\cdot ),\tau (\cdot )}$.
\end{theorem}

The convergence in $\mathcal{S}^{\prime }(\mathbb{R}^{n})$ can be obtained
as a by-product of the proof using the same method as in \cite[ Corollary
13.9 ]{T3} and \cite{D6}.\vskip5pt

If $p$, $q$, $\tau $, and $\alpha $ are constants, then the restriction $%
\mathrm{\eqref{K,L,B-F-cond}}$, and their counterparts, in the atomic
decomposition theorem are $K\geq ([\alpha +n\tau ]+1)^{+}$\ and $L\geq \max
(-1,[n(\frac{1}{\min (1,p,q)}-1)-\alpha ])$, which are essentially the
restrictions from the works of \cite[Theorem 3.12]{D4}.

\begin{proof}
{The proof follows the ideas in \cite[Theorem 6]{FJ86}.\vskip5pt }

\textit{Step 1}. Assume that $f\in \mathfrak{F}_{p(\cdot ),q(\cdot
)}^{\alpha (\cdot ),\tau (\cdot )}$ and let $\Phi $ and $\varphi $ satisfy,
respectively $\mathrm{\eqref{Ass1}}$ and $\mathrm{\eqref{Ass2}}$. There
exist \ functions $\Psi \in \mathcal{S}(\mathbb{R}^{n})$ satisfying $\mathrm{%
\eqref{Ass1}}$ and $\psi \in \mathcal{S}(\mathbb{R}^{n})$ satisfying $%
\mathrm{\eqref{Ass2}}$ such that for all $\xi \in \mathbb{R}^{n}$%
\begin{equation*}
f=\Psi \ast \widetilde{\Phi }\ast f+\sum_{v=1}^{\infty }\psi _{v}\ast 
\widetilde{\varphi }_{v}\ast f,
\end{equation*}%
see Section 3. Using the definition of the cubes $Q_{v,m}$ we obtain%
\begin{equation*}
f(x)=\sum\limits_{m\in \mathbb{Z}^{n}}\int_{Q_{0,m}}\widetilde{\Phi }%
(x-y)\Psi \ast f(y)dy+\sum_{v=1}^{\infty }2^{vn}\sum\limits_{m\in \mathbb{Z}%
^{n}}\int_{Q_{v,m}}\widetilde{\varphi }(2^{v}(x-y))\psi _{v}\ast f(y)dy,
\end{equation*}%
with convergence in $\mathcal{S}^{\prime }(\mathbb{R}^{n})$. We define for
every $v\in \mathbb{N}$ and all $m\in \mathbb{Z}^{n}$%
\begin{equation}
\lambda _{v,m}=C_{\theta }\sup_{y\in Q_{v,m}}\left\vert \psi _{v}\ast
f(y)\right\vert  \label{Coefficient}
\end{equation}%
where%
\begin{equation*}
C_{\varphi }=\max \{\sup_{\left\vert y\right\vert \leq 1}\left\vert
D^{\alpha }\varphi (y)\right\vert :\left\vert \alpha \right\vert \leq K\}.
\end{equation*}%
Define also%
\begin{equation}
\varrho _{v,m}(x)=\left\{ 
\begin{array}{ccc}
\frac{1}{\lambda _{v,m}}2^{vn}\int_{Q_{v,m}}\widetilde{\varphi }%
_{v}(2^{v}(x-y))\psi _{v}\ast f(y)dy & \text{if} & \lambda _{v,m}\neq 0 \\ 
0 & \text{if} & \lambda _{v,m}=0%
\end{array}%
\right. .  \label{K-L-atom}
\end{equation}%
Similarly we define for every $m\in \mathbb{Z}^{n}$ the numbers $\lambda
_{0,m}$ and the functions $\varrho _{0,m}$ taking in $\mathrm{%
\eqref{Coefficient}}$ and $\mathrm{\eqref{K-L-atom}}$ $v=0$ and replacing $%
\psi _{v}$ and $\widetilde{\varphi }$ by $\Psi $ and $\widetilde{\Phi }$,
respectively. Let us now check that such $\varrho _{v,m}$ are atoms in the
sense of Definition \ref{Atom-Def}. Note that the support and moment
conditions are clear by $\mathrm{\eqref{Ass1}}$ and $\mathrm{\eqref{Ass2}}$,
respectively. It thus remains to check $\mathrm{\eqref{diff-cond}}$ in
Definition \ref{Atom-Def}. We have%
\begin{eqnarray*}
\left\vert D^{\beta }\varrho _{v,m}(x)\right\vert &\leq &\frac{%
2^{v(n+\left\vert \beta \right\vert )}}{C_{\varphi }}\int_{Q_{v,m}}\left%
\vert (D^{\beta }\widetilde{\varphi })(2^{v}(x-y))\right\vert \left\vert
\psi _{v}\ast f(y)\right\vert dy\Big(\sup_{y\in Q_{v,m}}\left\vert \psi
_{v}\ast f(y)\right\vert \Big)^{-1} \\
&\leq &\frac{2^{v(n+\left\vert \beta \right\vert )}}{C_{\varphi }}%
\int_{Q_{v,m}}\left\vert (D^{\beta }\widetilde{\varphi })(2^{v}(x-y))\right%
\vert dy\leq 2^{v(n+\left\vert \beta \right\vert )}\left\vert
Q_{v,m}\right\vert \leq 2^{v\left\vert \beta \right\vert }.
\end{eqnarray*}%
The modifications for the terms with $v=0$ are obvious.\vskip5pt

\textit{Step }2\textit{.} Next we show that there is a constant $c>0$ such
that $\left\Vert \lambda \right\Vert _{\mathfrak{f}_{p(\cdot ),q(\cdot
)}^{\alpha (\cdot ),\tau (\cdot )}}\leq c\left\Vert f\right\Vert _{\mathfrak{%
F}_{p(\cdot ),q(\cdot )}^{\alpha (\cdot ),\tau (\cdot )}}$. For that reason
we exploit the equivalent quasi-norms given in Theorem \ref{fun-char}
involving Peetre's maximal function. Let $v\in \mathbb{N}$. Taking into
account that $\left\vert x-y\right\vert \leq c$ $2^{-v}$ for $x,y\in Q_{v,m}$
we obtain%
\begin{equation*}
2^{v(\alpha \left( x\right) -\alpha \left( y\right) )}\leq \frac{c_{\log
}(\alpha )v}{\log (e+1/\left\vert x-y\right\vert )}\leq \frac{c_{\log
}(\alpha )v}{\log (e+2^{v}/c)}\leq c
\end{equation*}%
if $v\geq \left[ \log _{2}c\right] +2$. If $0<v<\left[ \log _{2}c\right] +2$%
, then $2^{v(\alpha \left( x\right) -\alpha \left( y\right) )}\leq
2^{v(\alpha ^{+}-\alpha ^{-})}\leq c$. Therefore,%
\begin{equation*}
2^{v\alpha \left( x\right) }\left\vert \psi _{v}\ast f(y)\right\vert \leq c%
\text{ }2^{v\alpha \left( y\right) }\left\vert \psi _{v}\ast f(y)\right\vert
\end{equation*}%
for any $x,y\in Q_{v,m}$ and any $v\in \mathbb{N}$. Hence,%
\begin{eqnarray*}
\sum\limits_{m\in \mathbb{Z}^{n}}\lambda _{v,m}2^{v\alpha \left( x\right)
}\chi _{v,m}(x) &=&C_{\theta }\sum\limits_{m\in \mathbb{Z}^{n}}2^{v\alpha
\left( x\right) }\sup_{y\in Q_{v,m}}\left\vert \psi _{v}\ast f(y)\right\vert
\chi _{v,m}(x) \\
&\leq &c\sum\limits_{m\in \mathbb{Z}^{n}}\sup_{\left\vert z\right\vert \leq c%
\text{ }2^{-v}}\frac{2^{v\alpha (x-z)}\left\vert \psi _{v}\ast
f(x-z)\right\vert }{(1+2^{v}\left\vert z\right\vert )^{a}}(1+2^{v}\left\vert
z\right\vert )^{a}\chi _{v,m}(x) \\
&\leq &c\text{ }\psi _{v}^{\ast ,a}2^{v\alpha \left( \cdot \right)
}f(x)\sum\limits_{m\in \mathbb{Z}^{n}}\chi _{v,m}(x)=c\text{ }\psi
_{v}^{\ast ,a}2^{v\alpha \left( \cdot \right) }f(x),
\end{eqnarray*}%
where we have used $\sum\limits_{m\in \mathbb{Z}^{n}}\chi _{v,m}(x)=1$. This
estimate and its counterpart for $v=0$ (which can be obtained by a similar
calculation) give%
\begin{equation*}
\left\Vert \lambda \right\Vert _{\mathfrak{f}_{p(\cdot ),q(\cdot )}^{\alpha
(\cdot ),\tau (\cdot )}}\leq c\left\Vert \left( \psi _{v}^{\ast
,a}2^{v\alpha \left( \cdot \right) }f\right) _{v}\right\Vert _{L_{p(\cdot
)}^{\tau (\cdot )}\left( \ell ^{q(\cdot )}\right) }\leq c\left\Vert
f\right\Vert _{\mathfrak{F}_{p(\cdot ),q(\cdot )}^{\alpha (\cdot ),\tau
(\cdot )}},
\end{equation*}%
by Theorem \ref{fun-char}.\vskip5pt

\textit{Step }3\textit{.} Assume that $f$ can be represented by $\mathrm{%
\eqref{new-rep}}$, with $K$ and $L$ satisfying $\mathrm{\eqref{K,L,B-F-cond}}
$ and $\mathrm{\eqref{K,L,B-cond}}$, respectively. We will show that $f\in 
\mathfrak{F}_{p\left( \cdot \right) ,q\left( \cdot \right) }^{\alpha \left(
\cdot \right) ,\tau \left( \cdot \right) }$ and that for some $c>0$, $%
\left\Vert f\right\Vert _{\mathfrak{F}_{p\left( \cdot \right) ,q\left( \cdot
\right) }^{\alpha \left( \cdot \right) ,\tau \left( \cdot \right) }}\leq
c\left\Vert \lambda \right\Vert _{\mathfrak{f}_{p\left( \cdot \right)
,q\left( \cdot \right) }^{\alpha \left( \cdot \right) ,\tau \left( \cdot
\right) }}$. The arguments are very similar to those in {\cite{D6}. }For the
convenience of the reader, we give some details. We write 
\begin{equation*}
f=\sum\limits_{v=0}^{\infty }\sum\limits_{m\in \mathbb{Z}^{n}}\lambda
_{v,m}\varrho _{v,m}=\sum\limits_{v=0}^{j}\cdot \cdot \cdot
+\sum\limits_{v=j+1}^{\infty }\cdot \cdot \cdot .
\end{equation*}%
From Lemmas \ref{FJ-lemma} and \ref{Conv-est1}, we have for any $M$
sufficiently large and any $v\leq j$%
\begin{eqnarray*}
&&\sum\limits_{m\in \mathbb{Z}^{n}}2^{j\alpha \left( x\right) }\left\vert
\lambda _{v,m}\right\vert \left\vert \varphi _{j}\ast \varrho
_{v,m}(x)\right\vert \\
&\lesssim &2^{(v-j)(K-\alpha ^{+})}\sum\limits_{m\in \mathbb{Z}%
^{n}}2^{v(\alpha \left( x\right) -n/2)}\left\vert \lambda _{v,m}\right\vert
\eta _{v,M}(x-x_{Q_{v,m}}) \\
&\lesssim &2^{(v-j)(K-\alpha ^{+})}\sum\limits_{m\in \mathbb{Z}%
^{n}}2^{v(\alpha \left( x\right) +n/2)}\left\vert \lambda _{v,m}\right\vert
\eta _{v,M}\ast \chi _{v,m}(x).
\end{eqnarray*}%
Lemma \ref{DHR-lemma} gives $2^{v\alpha \left( \cdot \right) }\eta
_{v,M}\ast \chi _{v,m}\lesssim \eta _{v,T}\ast 2^{v\alpha \left( \cdot
\right) }\chi _{v,m}$, with $T=M-c_{\log }(\alpha )$ and since $K>\alpha
^{+}+\tau ^{+}$ we apply Lemma \ref{Key-lemma} to obtain%
\begin{eqnarray*}
&&\Big\|\Big(\sum_{v=0}^{j}2^{(v-j)(K-\alpha ^{+})}\eta _{v,T}\ast \Big[%
2^{v(\alpha \left( \cdot \right) +n/2)}\sum\limits_{m\in \mathbb{Z}%
^{n}}\left\vert \lambda _{v,m}\right\vert \chi _{v,m}\Big]\Big)_{j}\Big\|%
_{L_{p(\cdot )}^{\tau (\cdot )}\left( \ell ^{q(\cdot )}\right) } \\
&\lesssim &\Big\|\Big(\eta _{v,T}\ast \Big[2^{v(\alpha \left( \cdot \right)
+n/2)}\sum\limits_{m\in \mathbb{Z}^{n}}\left\vert \lambda _{v,m}\right\vert
\chi _{v,m}\Big]\Big)_{v}\Big\|_{L_{p(\cdot )}^{\tau (\cdot )}\left( \ell
^{q(\cdot )}\right) }.
\end{eqnarray*}%
The right-hand side can be rewritten us%
\begin{eqnarray*}
&&\sup_{P\in \mathcal{Q}}\Big\|\Big(\frac{\Big(\eta _{v,T}\ast \Big[%
2^{v(\alpha \left( \cdot \right) +n/2)}\sum\limits_{m\in \mathbb{Z}%
^{n}}\left\vert \lambda _{v,m}\right\vert \chi _{v,m}\Big]\Big)^{r}}{%
\left\vert P\right\vert ^{r\tau (\cdot )}}\chi _{P}\Big)_{v\geq v_{P}^{+}}%
\Big\|_{L^{p(\cdot )/r}(\ell ^{q(\cdot )/r})}^{1/r} \\
&\lesssim &\sup_{P\in \mathcal{Q}}\Big\|\Big(\frac{\eta _{v,Tr}\ast \Big[%
2^{v(\alpha \left( \cdot \right) +n/2)r}\sum\limits_{m\in \mathbb{Z}%
^{n}}\left\vert \lambda _{v,m}\right\vert ^{r}\chi _{v,m}\Big]}{\left\vert
P\right\vert ^{r\tau (\cdot )}}\chi _{P}\Big)_{v\geq v_{P}^{+}}\Big\|%
_{L^{p(\cdot )/r}(\ell ^{q(\cdot )/r})}^{1/r},
\end{eqnarray*}%
by Lemma \ref{Conv-est2}, since $\eta _{v,T}\approx \eta _{v,T}\ast \eta
_{v,T}$ and $0<r<\frac{\tau ^{-}}{\tau ^{+}}\min (1,p^{-},q^{-})$. The
application of Lemma \ref{DHR-lemma1} and the fact that 
\begin{equation*}
\left\Vert \left( g_{v}\right) _{v\geq v_{P}^{+}}\right\Vert _{L^{p(\cdot
)/r}(\ell ^{q(\cdot )/r})}^{1/r}=\left\Vert \left( |g_{v}|^{1/r}\right)
_{v\geq v_{P}^{+}}\right\Vert _{L^{p(\cdot )}(\ell ^{q(\cdot )})}
\end{equation*}%
give that the last expression is bounded by $\left\Vert \lambda \right\Vert
_{\mathfrak{f}_{p\left( \cdot \right) ,q\left( \cdot \right) }^{\alpha
\left( \cdot \right) ,\tau \left( \cdot \right) }}$. Now from Lemma \ref%
{FJ-lemma}, we have for any $M$ sufficiently large and $v\geq j$%
\begin{eqnarray*}
&&\sum\limits_{m\in \mathbb{Z}^{n}}2^{j\alpha \left( x\right) }\left\vert
\lambda _{v,m}\right\vert \left\vert \varphi _{j}\ast \varrho
_{v,m}(x)\right\vert \\
&\lesssim &2^{(j-v)(L+1+n/2)}\sum\limits_{m\in \mathbb{Z}^{n}}2^{j(\alpha
\left( x\right) -n/2)}\left\vert \lambda _{v,m}\right\vert \eta
_{j,M}(x-x_{Q_{v,m}}) \\
&\lesssim &2^{(j-v)(L+1+n/2)}\sum\limits_{m\in \mathbb{Z}^{n}}2^{j(\alpha
\left( x\right) -n/2)}\left\vert \lambda _{v,m}\right\vert \eta _{j,M}\ast
\eta _{v,M}(x-x_{Q_{v,m}}),
\end{eqnarray*}%
where the last inequality follows by Lemma \ref{Conv-est}, since $\eta
_{j,M}=\eta _{\min (v,j),M}$. Again by Lemma \ref{Conv-est1}, we have%
\begin{equation*}
\eta _{j,M}\ast \eta _{v,M}(x-x_{Q_{v,m}})\lesssim 2^{vn}\eta _{j,M}\ast
\eta _{v,M}\ast \chi _{v,m}(x).
\end{equation*}%
Therefore, $\sum\limits_{m\in \mathbb{Z}^{n}}2^{j\alpha \left( x\right)
}\left\vert \lambda _{v,m}\right\vert \left\vert \varphi _{j}\ast \varrho
_{v,m}(x)\right\vert $ is bounded by%
\begin{eqnarray*}
&&c\text{ }2^{(j-v)(L+1-n/2)}\sum\limits_{m\in \mathbb{Z}^{n}}2^{j(\alpha
\left( x\right) +n/2)}\left\vert \lambda _{v,m}\right\vert \eta _{j,M}\ast
\eta _{v,M}\ast \chi _{v,m}(x) \\
&\lesssim &2^{(j-v)(L+1-\alpha ^{-})}\eta _{j,T}\ast \eta _{v,T}\ast \Big[%
2^{v(\alpha \left( \cdot \right) +n/2)}\sum\limits_{m\in \mathbb{Z}%
^{n}}\left\vert \lambda _{v,m}\right\vert \chi _{v,m}\Big](x),
\end{eqnarray*}%
by Lemma \ref{DHR-lemma}, with $T=M-c_{\log }(\alpha )$. Let $0<r<\frac{\tau
^{-}}{\tau ^{+}}\min (1,p^{-},q^{-})$\ be a real number such that $%
L>n/r-1-\alpha ^{-}-n$. We have%
\begin{eqnarray*}
&&\Big(\sum_{v=j}^{\infty }2^{(j-v)(L+1-\alpha ^{-})}\eta _{j,T}\ast \eta
_{v,T}\ast \Big[2^{v(\alpha \left( \cdot \right) +n/2)}\sum\limits_{m\in 
\mathbb{Z}^{n}}\left\vert \lambda _{v,m}\right\vert \chi _{v,m}\Big]\Big)^{r}
\\
&\leq &\sum_{v=j}^{\infty }2^{(j-v)(L-n/r+1-\alpha ^{-}+n)r}\eta _{j,Tr}\ast
\eta _{v,Tr}\ast \Big[2^{v(\alpha \left( \cdot \right)
+n/2)r}\sum\limits_{m\in \mathbb{Z}^{n}}\left\vert \lambda _{v,m}\right\vert
^{r}\chi _{v,m}\Big],
\end{eqnarray*}%
where we have used Lemma \ref{Conv-est2}. The application of Lemma \ref%
{DHR-lemma1} gives that%
\begin{equation*}
\Big\|\Big(\sum_{v=j}^{\infty }2^{(j-v)(L+1-\alpha ^{-})}\eta _{j,T}\ast
\eta _{v,T}\ast \Big[2^{v(\alpha \left( \cdot \right)
+n/2)}\sum\limits_{m\in \mathbb{Z}^{n}}\left\vert \lambda _{v,m}\right\vert
\chi _{v,m}\Big]\Big)_{j}\Big\|_{L_{p(\cdot )}^{\tau (\cdot )}\left( \ell
^{q(\cdot )}\right) }
\end{equation*}%
is bounded by%
\begin{equation*}
c\sup_{P\in \mathcal{Q}}\Big\|\Big(\frac{\sum_{v=j}^{\infty }2^{(j-v)Hr}\eta
_{v,Tr}\ast \Big[2^{v(\alpha \left( \cdot \right) +n/2)r}\sum\limits_{m\in 
\mathbb{Z}^{n}}\left\vert \lambda _{vm}\right\vert ^{r}\chi _{v,m}\Big]}{%
\left\vert P\right\vert ^{r\tau (\cdot )}}\chi _{P}\Big)_{j\geq j_{P}^{+}}%
\Big\|_{L^{p(\cdot )/r}(\ell ^{q(\cdot )/r})}^{1/r},
\end{equation*}%
where $H:=L-n/r+n+1-\alpha ^{-}$. Observing that $H>0$, an application of
Lemma \ref{Key-lemma},\ yields that the last expression is bounded by%
\begin{equation*}
c\sup_{P\in \mathcal{Q}}\Big\|\Big(\frac{\eta _{v,Tr}\ast \Big[2^{v(\alpha
\left( \cdot \right) +n/2)r}\sum\limits_{m\in \mathbb{Z}^{n}}\left\vert
\lambda _{v,m}\right\vert ^{r}\chi _{v,m}\Big]}{\left\vert P\right\vert
^{r\tau (\cdot )}}\chi _{P}\Big)_{v\geq v_{P}^{+}}\Big\|_{L^{p(\cdot
)/r}(\ell ^{q(\cdot )/r})}^{1/r}\lesssim \left\Vert \lambda \right\Vert _{%
\mathfrak{f}_{p\left( \cdot \right) ,q\left( \cdot \right) }^{\alpha \left(
\cdot \right) ,\tau \left( \cdot \right) }},
\end{equation*}%
where we used again Lemma \ref{DHR-lemma1} and hence the proof is complete.
\end{proof}

\section{Appendix}

Here we present more technical proofs of the Lemmas.\vskip5pt

{\emph{Proof of Lemma} \ref{DHR-lemma1}.} By the scaling argument, we see
that it suffices to consider when $\left\Vert \left( f_{v}\right)
_{v}\right\Vert _{L_{p(\cdot )}^{\tau (\cdot )}\left( \ell ^{q(\cdot
)}\right) }\leq 1$\ and show that for any dyadic cube $P$ 
\begin{equation}
\Big\|\Big(\sum_{v=v_{P}^{+}}^{\infty }\left\vert \frac{\eta _{v,m}\ast
\left\vert f_{v}\right\vert }{|P|^{\tau (\cdot )}}\right\vert ^{q(\cdot )}%
\Big)^{1/q(\cdot )}\chi _{P}\Big\|_{p(\cdot )}\lesssim 1.  \label{key-est3}
\end{equation}

\textit{Case 1.} $|P|>1$. Let $Q_{v}\subset P$ be a cube, with $\ell \left(
Q_{v}\right) =2^{-v}$ and $x\in Q_{v}\subset P$. We have%
\begin{eqnarray}
&&\eta _{v,m}\ast \left\vert f_{v}\right\vert (x)  \notag \\
&=&2^{vn}\int_{\mathbb{R}^{n}}\frac{\left\vert f_{v}(z)\right\vert }{\left(
1+2^{v}\left\vert x-z\right\vert \right) ^{m}}dz  \notag \\
&=&\int_{3Q_{v}}\cdot \cdot \cdot dz+\sum_{k=(k_{1},...,k_{n})\in \mathbb{Z}%
^{n},\max_{i=1,...,n}|k_{i}|\geq 2}\int_{Q_{v}+kl(Q_{v})}\cdot \cdot \cdot dz
\notag \\
&=&J_{v}^{1}(f_{v}\chi _{3Q_{v}})(x)+\sum_{k=(k_{1},...,k_{n})\in \mathbb{Z}%
^{n},\max_{i=1,...,n}|k_{i}|\geq 2}J_{v,k}^{2}(f_{v}\chi
_{Q_{v}+kl(Q_{v})})(x).  \label{sum1}
\end{eqnarray}%
Let $z\in Q_{v}+kl(Q_{v})$ with $k\in \mathbb{Z}^{n}$ and $|k|>4\sqrt{n}$.
Then $\left\vert x-z\right\vert \geq \left\vert k\right\vert 2^{-v-1}$ and 
\begin{equation*}
J_{v,k}^{2}(f_{v}\chi _{Q_{v}+kl(Q_{v})})(x)\lesssim \left\vert k\right\vert
^{-m}M_{Q_{v}+kl(Q_{v})}\left( f_{v}\right) (x).
\end{equation*}%
Let $d>0$ be such that $\tau ^{+}<d\leq \tau ^{-}\min \left(
p^{-},q^{-}\right) $. Therefore, the left-hand side of $\mathrm{%
\eqref{key-est3}}$ is bounded\ by%
\begin{eqnarray}
&&\sum_{k\in \mathbb{Z}^{n},|k|\leq 4\sqrt{n}}\Big\|\Big(\sum_{v=0}^{\infty
}\left\vert \frac{g_{v,k}}{|P|^{\tau (\cdot )}}\right\vert ^{q(\cdot )}\Big)%
^{1/q(\cdot )}\chi _{P}\Big\|_{p(\cdot )}  \label{sum-est} \\
&&+\sum_{k\in \mathbb{Z}^{n},|k|>4\sqrt{n}}\left\vert k\right\vert
^{w-m+n\left( 1+\frac{1}{t^{-}}\right) \tau ^{+}}\log \left\vert
k\right\vert \Big\|\Big(\sum_{v=0}^{\infty }\left\vert \frac{g_{v,k}}{%
|P|^{\tau (\cdot )}}\right\vert ^{q(\cdot )}\Big)^{1/q(\cdot )}\chi _{P}%
\Big\|_{p(\cdot )},  \notag
\end{eqnarray}%
where 
\begin{equation*}
g_{v,k}=\left\{ 
\begin{array}{ccc}
M_{3Q_{v}}\left( f_{v}\right) & \text{if} & k=0 \\ 
M_{Q_{v}+kl(Q_{v})}\left( f_{v}\right) & \text{if} & 0<|k|\leq 4\sqrt{n} \\ 
M_{Q_{v}+kl(Q_{v})}\left( \left\vert k\right\vert ^{-n\left( 1+1/t\left(
\cdot \right) \right) \tau \left( \cdot \right) }f_{v}\right) & \text{if} & 
|k|>4\sqrt{n}.%
\end{array}%
\right.
\end{equation*}%
with $\frac{1}{d}=\frac{1}{p(\cdot )\tau (\cdot )}+\frac{1}{t(\cdot )}$. By
similarity we only estimate the second norm in $\mathrm{\eqref{sum-est}}$.
This term is bounded if and only if%
\begin{equation}
\Big\|\Big(\sum_{v=0}^{\infty }\Big|\frac{(b_{k}g_{v,k})^{d/\tau (\cdot )}}{%
|P|^{d}}\Big|^{q(\cdot )\tau (\cdot )/d}\Big)^{d/\tau (\cdot )q(\cdot )}\chi
_{P}\Big\|_{p(\cdot )\tau (\cdot )/d}\lesssim 1,  \label{key-est4}
\end{equation}%
where $b_{k}=\left( \frac{1}{|k|^{w}\log \left\vert k\right\vert }\right)
^{\tau (\cdot )/d}$. Observe that $Q_{v}+kl(Q_{v})\subset Q(x,\left\vert
k\right\vert 2^{1-v})$. By H\"{o}lder's inequality, 
\begin{eqnarray*}
&&\left\vert Q(x,\left\vert k\right\vert 2^{1-v})\right\vert
M_{Q(x,\left\vert k\right\vert 2^{1-v})}\left( \left\vert k\right\vert
^{-n\left( 1+1/t\left( \cdot \right) \right) d}\left\vert f_{v}\right\vert
^{d/\tau \left( \cdot \right) }\right) \left( x\right) \\
&\lesssim &\Big\|\frac{\left\vert f_{v}\right\vert ^{1/\tau (\cdot )}}{%
|Q(x,\left\vert k\right\vert 2^{1-v})|}\chi _{Q(x,\left\vert k\right\vert
2^{1-v})}\Big\|_{p(\cdot )\tau (\cdot )}^{d}\Big\||Q(x,\left\vert
k\right\vert 2^{1-v})|\left\vert k\right\vert ^{-n\left( 1+1/t(\cdot
)\right) }\chi _{Q(x,\left\vert k\right\vert 2^{1-v})}\Big\|_{t(\cdot )}^{d}
\end{eqnarray*}%
for any $x\in Q(x,\left\vert k\right\vert 2^{1-v})$. The second norm is
bounded. The first norm is bounded if and only if%
\begin{equation*}
\Big\|\frac{f_{v}}{|Q(x,\left\vert k\right\vert 2^{1-v})|^{\tau (\cdot )}}%
\chi _{Q(x,\left\vert k\right\vert 2^{1-v})}\Big\|_{p(\cdot )}\lesssim 1,
\end{equation*}%
{which follows since }$\left\Vert \left( f_{v}\right) _{v}\right\Vert
_{L_{p(\cdot )}^{\tau (\cdot )}\left( \ell ^{q(\cdot )}\right) }\leq 1$.
Therefore we can apply Lemma \ref{DHHR-estimate}, 
\begin{eqnarray*}
&&\Big(M_{Q(x,\left\vert k\right\vert 2^{1-v})}\left( \left\vert
k\right\vert ^{-n\left( 1+1/t\left( \cdot \right) \right) \tau \left( \cdot
\right) }f_{v}\right) \left( x\right) \Big)^{d/\tau \left( x\right) } \\
&\leq &M_{Q(x,\left\vert k\right\vert 2^{1-v})}\Big(\left\vert \left\vert
k\right\vert ^{-n\left( 1+1/t\left( \cdot \right) \right) \tau \left( \cdot
\right) }f_{v}\right\vert ^{d/\tau \left( \cdot \right) }\Big)\left(
x\right) +\min \left( 1,\left\vert k\right\vert ^{ns}2^{n\left( 1-v\right)
s}\right) \omega (x)
\end{eqnarray*}%
for any $s>0$ large enough, where%
\begin{eqnarray*}
\omega (x) &=&\left( e+\left\vert x\right\vert \right)
^{-s}+M_{Q(x,\left\vert k\right\vert 2^{1-v})}\left( \left( e+\left\vert
\cdot \right\vert \right) ^{-s}\right) \left( x\right) \\
&\leq &\left( e+\left\vert x\right\vert \right) ^{-s}+\mathcal{M}\left(
\left( e+\left\vert \cdot \right\vert \right) ^{-s}\right) \left( x\right)
=h\left( x\right) .
\end{eqnarray*}%
Therefore $\left( g_{v,k}\right) ^{d/\tau (\cdot )}$ can be estimated by%
\begin{equation*}
c\text{ }M_{Q(\cdot ,\left\vert k\right\vert 2^{1-v})}\left( \left\vert
k\right\vert ^{-nd}\left\vert f_{v}\right\vert ^{d/\tau \left( \cdot \right)
}\right) +\sigma _{v,k}\text{ }h,
\end{equation*}%
where%
\begin{equation*}
\sigma _{v,k}=\left\{ 
\begin{array}{ccc}
1 & \text{if} & 2^{v}\leq 2\left\vert k\right\vert \\ 
\left\vert k\right\vert ^{ns}2^{-vns} & \text{if} & 2^{v}>2\left\vert
k\right\vert .%
\end{array}%
\right.
\end{equation*}%
Therefore, the quantity $\Big(\cdot \cdot \cdot \Big)^{d/\tau (\cdot
)q(\cdot )}$ of the term $\mathrm{\eqref{key-est4}}$ is bounded by 
\begin{eqnarray}
&&|k|^{-w}\Big(\sum_{v=0}^{\infty }\Big|M_{Q(\cdot ,\left\vert k\right\vert
2^{-v+1})}\Big(\frac{\left\vert k\right\vert ^{-nd}\left\vert
f_{v}\right\vert ^{d/\tau \left( \cdot \right) }}{|P|^{d}}\chi _{Q(\cdot
,\left\vert k\right\vert 2^{-v+1})}\Big)\Big|^{q(\cdot )\tau (\cdot )/d}\Big)%
^{d/\tau (\cdot )q(\cdot )}  \notag \\
&&+\frac{1}{\log \left\vert k\right\vert }\Big(\sum_{v=0}^{\infty }\left(
\sigma _{v,k}\left\vert h\right\vert \right) ^{q(\cdot )\tau (\cdot )/d}\Big)%
^{d/\tau (\cdot )q(\cdot )}.  \label{sum3}
\end{eqnarray}%
The first term is bounded by%
\begin{equation*}
\Big(\sum_{v=0}^{\infty }\Big|\eta _{v,w}\ast \Big(\frac{\left\vert
k\right\vert ^{-nd}\left\vert f_{v}\right\vert ^{d/\tau \left( \cdot \right)
}}{|P|^{d}}\chi _{Q(\cdot ,\left\vert k\right\vert 2^{1-v})}\Big)\Big|%
^{q(\cdot )\tau (\cdot )/d}\Big)^{d/\tau (\cdot )q(\cdot )}.
\end{equation*}%
where $w>n+c_{\log }(1/q)+c_{\log }(\tau )$. Applying Theorem \ref%
{DHR-theorem}, the $L^{p(\cdot )\tau (\cdot )/d}$-norm\ of this expression
is bounded by%
\begin{equation*}
\Big\|\Big(\sum_{v=0}^{\infty }\Big|\frac{\left\vert k\right\vert ^{-n\tau
(\cdot )}f_{v}}{|P|^{\tau (\cdot )}}\Big|^{q(\cdot )}\chi _{Q(\cdot
,\left\vert k\right\vert 2^{1-v})}\Big)^{d/\tau (\cdot )q(\cdot )}\Big\|%
_{p(\cdot )\tau (\cdot )/d}.
\end{equation*}%
Observe that $Q(\cdot ,\left\vert k\right\vert 2^{-v+1})\subset
Q(c_{P},\left\vert k\right\vert 2^{-v_{P}+1})$ and the measure of the last
cube is greater that $1$, the last norm is bounded by $1$. The summation $%
\mathrm{\eqref{sum3}}$ can be estimated by%
\begin{eqnarray*}
&&c\left\vert h\right\vert ^{q(\cdot )\tau (\cdot )/d}\Big(c+\sum_{v\geq
1,2^{v}\leq 2\left\vert k\right\vert }\frac{1}{\log \left\vert k\right\vert }%
+\sum_{2^{v}>2\left\vert k\right\vert }\Big(\frac{2^{v}}{\left\vert
k\right\vert }\Big)^{-nsq(\cdot )\tau (\cdot )/d}\Big) \\
&\lesssim &\left\vert h\right\vert ^{q(\cdot )\tau (\cdot )/d}.
\end{eqnarray*}%
This expression, with power $d/\tau (\cdot )q(\cdot )$, in the $L^{p(\cdot
)\tau (\cdot )/d}$-norm is bounded by $1$. Therefore, the second sum in $%
\mathrm{\eqref{sum-est}}$ is bounded by taking $m$ large enough such that $%
m>n\tau ^{+}+2n+w$, with $w>n+c_{\log }(1/q)+c_{\log }(\tau )$.\vskip5pt

\textit{Case 2.} $|P|\leq 1$. As before, 
\begin{equation*}
\eta _{v,m}\ast \left\vert f_{v}\right\vert (x)\lesssim J_{v}^{1}(f_{v}\chi
_{3P})(x)+\sum_{k=(k_{1},...,k_{n})\in \mathbb{Z}^{n},%
\max_{i=1,...,n}|k_{i}|\geq 2}J_{v,k}^{2}(f_{v}\chi _{P+kl(P)})(x).
\end{equation*}%
We see that 
\begin{equation*}
J_{v}^{1}(f_{v}\chi _{3P})(x)\lesssim \eta _{v,m}\ast \left( \left\vert
f_{v}\right\vert \chi _{3P}\right) (x).
\end{equation*}%
S{ince }$\tau ${\ is log-H\"{o}lder continuous,}%
\begin{equation*}
\left\vert P\right\vert ^{-\tau (x)}\leq c\left\vert P\right\vert ^{-\tau
(y)}(1+2^{v_{P}}\left\vert x-y\right\vert )^{c_{\log }\left( \tau \right)
}\leq c\left\vert P\right\vert ^{-\tau (y)}(1+2^{v}\left\vert x-y\right\vert
)^{c_{\log }\left( \tau \right) }
\end{equation*}%
for any $x\in P$ and any $y\in 3P$. Hence%
\begin{equation*}
\left\vert P\right\vert ^{-\tau (x)}J_{v}^{1}(f_{v}\chi _{3P})(x)\lesssim
\eta _{v,m-c_{\log }\left( \tau \right) }\ast \left( \left\vert P\right\vert
^{-\tau (\cdot )}\left\vert f_{v}\right\vert \chi _{3P}\right) (x).
\end{equation*}%
Also, we have 
\begin{equation*}
\left\vert P\right\vert ^{-\tau (x)}J_{v,k}^{2}(f_{v}\chi
_{P+kl(P)})(x)\lesssim \eta _{v,m-c_{\log }\left( \tau \right) }\ast \left(
\left\vert P\right\vert ^{-\tau (\cdot )}\left\vert f_{v}\right\vert \chi
_{P+kl(P)}\right) (x)
\end{equation*}%
for $x\in P$ and $z\in P+kl(P)$ with $k\in \mathbb{Z}^{n}$ and $|k|>4\sqrt{n}
$. Our estimate follows by\ Lemma \ref{DHR-lemma}. The proof is complete.%
\vskip5pt

{\emph{Proof of Lemma }\ref{lamda-equi}. Obviously, $\left\Vert \lambda
\right\Vert _{\mathfrak{f}_{p\left( \cdot \right) ,q\left( \cdot \right)
}^{\alpha \left( \cdot \right) ,\tau (\cdot )}}\leq \left\Vert \lambda
_{r,d}^{\ast }\right\Vert _{\mathfrak{f}_{p\left( \cdot \right) ,q\left(
\cdot \right) }^{\alpha \left( \cdot \right) ,\tau (\cdot )}}$. Let us prove
the converse inequality. By the scaling argument, it suffices to consider
the case $\left\Vert \lambda \right\Vert _{\mathfrak{f}_{p\left( \cdot
\right) ,q\left( \cdot \right) }^{\alpha \left( \cdot \right) ,\tau (\cdot
)}}\leq 1$ and show that the modular of a the sequence on the left-hand side
is bounded. It suffices to prove that%
\begin{equation}
\Big\|\Big(\dsum\limits_{v=v_{P}^{+}}^{\infty }\Big|\frac{\sum\limits_{m\in 
\mathbb{Z}^{n}}2^{v(\alpha \left( \cdot \right) +n/2)}\lambda
_{v,m,r,d}^{\ast }\chi _{v,m}}{|P|^{\tau \left( \cdot \right) }}\Big|%
^{q(\cdot )}\Big)^{1/q(\cdot )}\chi _{P}\Big\|_{p(\cdot )}\lesssim 1
\label{key-est11}
\end{equation}%
for any dyadic cube $P\in \mathcal{Q}$}. {For each $k\in \mathbb{N}_{0}$ we
define $\Omega _{k}:=\{h\in \mathbb{Z}^{n}:2^{k-1}<2^{v}\left\vert
2^{-v}h-2^{-v}m\right\vert \leq 2^{k}\}$\ and $\Omega _{0}:=\{h\in \mathbb{Z}%
^{n}:2^{v}\left\vert 2^{-v}h-2^{-v}m\right\vert \leq 1\}$. Then for any $%
x\in Q_{v,m}\cap P$, $\sum_{h\in \mathbb{Z}^{n}}\frac{2^{vr\alpha \left(
x\right) }|\lambda _{v,h}|^{r}}{(1+2^{v}|2^{-v}h-2^{-v}m|)^{d}}$ can be
rewritten as%
\begin{eqnarray}
&&\sum\limits_{k=0}^{\infty }\sum\limits_{h\in \Omega _{k}}\frac{2^{vr\alpha
\left( x\right) }\left\vert \lambda _{v,h}\right\vert ^{r}}{\left(
1+2^{v}\left\vert 2^{-v}h-2^{-v}m\right\vert \right) ^{d}}  \notag \\
&\lesssim &\sum\limits_{k=0}^{\infty }2^{-dk}\sum\limits_{h\in \Omega
_{k}}2^{vr\alpha \left( x\right) }\left\vert \lambda _{v,h}\right\vert ^{r} 
\notag \\
&=&\sum\limits_{k=0}^{\infty }2^{(n-d)k+(v-k)n+vr\alpha \left( x\right)
}\int\limits_{\cup _{z\in \Omega _{k}}Q_{v,z}}\sum\limits_{h\in \Omega
_{k}}\left\vert \lambda _{v,h}\right\vert ^{r}\chi _{v,h}(y)dy.
\label{key-est2}
\end{eqnarray}%
Let $x\in Q_{v,m}\cap P$ and $y\in \cup _{z\in \Omega _{k}}Q_{v,z}$, then $%
y\in Q_{v,z}$ for some $z\in \Omega _{k}$ and $2^{k-1}<2^{v}\left\vert
2^{-v}z-2^{-v}m\right\vert \leq 2^{k}$. From this it follows that $y$ is
located in some cube $Q(x,2^{k-v+3})$. In addition, from the fact that%
\begin{equation*}
\left\vert y_{i}-(c_{P})_{i}\right\vert \leq \left\vert
y_{i}-x_{i}\right\vert +\left\vert x_{i}-(c_{P})_{i}\right\vert \leq
2^{k-v+2}+2^{-v_{P}-1}<2^{k-v_{P}+3},\text{ \ }i=1,...,n,
\end{equation*}%
we have $y$ is located in some cube $Q(c_{P},2^{k-v_{P}+4})$. S}ince $\alpha 
$ is log-H\"{o}lder continuous we can prove that 
\begin{equation*}
2^{v(\alpha \left( x\right) -\alpha \left( y\right) )}\lesssim \left\{ 
\begin{array}{ccc}
2^{c_{\log }(\alpha )k} & \text{if} & k<\max (0,v-h_{n}) \\ 
2^{(\alpha ^{+}-\alpha ^{-})k} & \text{if} & k\geq \max (0,v-h_{n}),%
\end{array}%
\right.
\end{equation*}%
where $c>0$ not depending on $v$ and $k$. Therefore, $\mathrm{%
\eqref{key-est2}}$ does not exceed%
\begin{eqnarray*}
&&c\sum\limits_{k=0}^{\infty
}2^{(n-d+a)k+(v-k)n}\int\limits_{Q(x,2^{k-v+3})}2^{v\alpha \left( y\right)
r}\sum\limits_{h\in \Omega _{k}}\left\vert \lambda _{v,h}\right\vert
^{r}\chi _{v,h}(y)\chi _{{Q(c_{P},2^{k-v_{P}+4})}}(y)dy \\
&\lesssim &\sum\limits_{k=0}^{\infty }2^{(n-d+a)k}M_{{Q(x,2^{k-v+3})}}\Big(%
\sum\limits_{h\in \Omega _{k}}2^{v\alpha \left( \cdot \right) r}\left\vert
\lambda _{v,h}\right\vert ^{r}\chi _{v,h}\chi _{{Q(c_{P},2^{k-v_{P}+4})}}%
\Big)(x),
\end{eqnarray*}%
where $a=\max \left( c_{\log }(\alpha ),\alpha ^{+}-\alpha ^{-}\right) $. To
prove\ {$\mathrm{\eqref{key-est11}}$\ }we can distinguish two cases:\vskip5pt

\textit{Case 1.} $|P|>1$. The left-hand side of $\mathrm{\eqref{key-est11}}$
is bounded by $1$ if and only if 
\begin{equation}
c\sum\limits_{k=0}^{\infty }2^{\varrho k}\Big\|\Big(\dsum\limits_{v=0}^{%
\infty }\Big(\frac{M_{{Q(\cdot ,2^{k-v+3})}}(2^{-kn\left( r+1/t\left( \cdot
\right) \right) \tau \left( \cdot \right) }g_{v,k,v_{P}})}{|P|^{r\tau \left(
\cdot \right) }}\Big)^{q\left( \cdot \right) /r}\Big)^{r/q\left( \cdot
\right) }\chi _{P}\Big\|_{p(\cdot )/r}\lesssim 1,  \label{principal-est}
\end{equation}%
{with}%
\begin{equation*}
g_{v,k,v_{P}}=\sum\limits_{h\in \Omega _{k}}2^{v(\alpha \left( \cdot \right)
+n/2)r}\left\vert \lambda _{v,h}\right\vert ^{r}\chi _{v,h}\chi _{{%
Q(c_{P},2^{k-v_{P}+4})}}
\end{equation*}%
and 
\begin{equation*}
\varrho =n-d+a+nr\tau ^{+}+n\frac{\tau ^{+}}{t^{-}}
\end{equation*}%
Let us prove that%
\begin{equation*}
\Big\|\Big(\dsum\limits_{v=0}^{\infty }\Big(\frac{\omega _{k}M_{{Q(\cdot
,2^{k-v+3})}}(2^{-kn\left( r+1/t\left( \cdot \right) \right) \tau \left(
\cdot \right) }g_{v,k,v_{P}})}{|P|^{r\tau \left( \cdot \right) }}\Big)%
^{q\left( \cdot \right) /r}\Big)^{r/q\left( \cdot \right) }\chi _{P}\Big\|%
_{p(\cdot )/r}\lesssim 1
\end{equation*}%
for any $k,v\in \mathbb{N}_{0}$ and any {$P\in \mathcal{Q}$, where}%
\begin{equation*}
\omega _{k}=\frac{1}{2^{(w-\frac{n\sigma }{t^{+}})k}+2^{kns}}.
\end{equation*}%
This is equivalent to%
\begin{equation}
\Big\|\Big(\dsum\limits_{v=0}^{\infty }\Big(\frac{\left( \omega _{k}M_{{%
Q(\cdot ,2^{k-v+3})}}(2^{-kn\left( r+1/t\left( \cdot \right) \right) \tau
\left( \cdot \right) }g_{v,k,v_{P}})\right) ^{\frac{\sigma }{\tau \left(
\cdot \right) }}}{|P|^{r\sigma }}\Big)^{\frac{q\left( \cdot \right) \tau
\left( \cdot \right) }{r\sigma }}\Big)^{\frac{\sigma r}{q\left( \cdot
\right) \tau \left( \cdot \right) }}\Big\|_{\frac{p(\cdot )\tau \left( \cdot
\right) }{r\sigma }}\lesssim 1,  \label{new-key-est}
\end{equation}%
where $\sigma >0$\ such that $\tau ^{+}<\sigma <\frac{\tau ^{-}\min \left(
p^{-},q^{-}\right) }{r}$. By H\"{o}lder's inequality, 
\begin{eqnarray*}
&&|{Q(x,2^{k-v+3})}|M_{{Q(x,2^{k-v+3})}}\left( 2^{-kn\left( r+1/t\left(
\cdot \right) \right) \sigma }\left\vert g_{v,k,v_{P}}\right\vert ^{\sigma
/\tau \left( \cdot \right) }\right) \left( x\right) \\
&\lesssim &\Big\|\frac{\left\vert g_{v,k,v_{P}}\right\vert ^{1/\tau (\cdot )}%
}{|{Q(x,2^{k-v+3})}|^{r}}\chi _{{Q(x,2^{k-v+3})}}\Big\|_{p(\cdot )\tau
(\cdot )/r}^{\sigma }\Big\||{Q(x,2^{k-v+3})}|^{r}2^{-kn\left( r+1/t\left(
\cdot \right) \right) }\chi _{{Q(x,2^{k-v+3})}}\Big\|_{t(\cdot )}^{\sigma }
\end{eqnarray*}%
for any $x\in {Q(x,2^{k-v+3})}$, where $\frac{1}{\sigma }=\frac{r}{p(\cdot
)\tau (\cdot )}+\frac{1}{t(\cdot )}$. The second norm is bounded. The first
norm is bounded if and only if%
\begin{equation*}
\left\Vert \frac{\left( g_{v,k,v_{P}}\right) ^{1/r}}{|{Q(x,2^{k-v+3})}%
|^{\tau (\cdot )}}\chi _{{Q(x,2^{k-v+3})}}\right\Vert _{p(\cdot )}\lesssim 1,
\end{equation*}%
{which follows since $\left\Vert \lambda \right\Vert _{\mathfrak{f}_{p\left(
\cdot \right) ,q\left( \cdot \right) }^{\alpha \left( \cdot \right) ,\tau
(\cdot )}}\leq 1$}. Hence we can apply Lemma \ref{DHHR-estimate} to estimate 
\begin{equation*}
\left( M_{{Q(x,2^{k-v+3})}}(2^{-kn\left( r+1/t\left( \cdot \right) \right)
\tau \left( \cdot \right) }g_{v,k,v_{P}})\left( x\right) \right) ^{\sigma
/\tau \left( x\right) }
\end{equation*}%
by%
\begin{equation*}
c\text{ }M_{{Q(x,2^{k-v+3})}}\left( \left( 2^{-kn\left( r+1/t\left( \cdot
\right) \right) \tau \left( \cdot \right) }g_{v,k,v_{P}}\right) ^{\sigma
/\tau \left( \cdot \right) }\right) \left( x\right) +\min \left( {2^{n(k-v)s}%
},1\right) h\left( x\right) ,
\end{equation*}%
where $s>0$\ large enough and $h$ is the same function in Lemma {\ref%
{DHR-lemma1}}. Hence the term in {$\mathrm{\eqref{new-key-est}}$, with }$%
\dsum\limits_{v=0}^{k+3}$ in place of $\dsum\limits_{v=0}^{\infty }$ is
bounded by {\ } 
\begin{eqnarray*}
&&c\dsum\limits_{v=0}^{k+3}\frac{1}{2^{(w-\frac{n\sigma }{t^{+}})k}+2^{kns}}%
\Big\|\frac{M_{{Q(\cdot ,2^{k-v+3})}}\left( g_{v,k,v_{P}}\right) ^{\frac{%
\sigma }{\tau \left( \cdot \right) }}}{|{Q(\cdot ,2^{k-v_{P}+3})}|^{r\sigma }%
}\Big\|_{\frac{p(\cdot )\tau \left( \cdot \right) }{r\sigma }} \\
&&+\frac{c\left( k+4\right) ^{s}}{2^{(w-\frac{n\sigma }{t^{+}})k}+2^{kns}}.
\end{eqnarray*}%
Since $\mathcal{M}$ is bounded in $L^{p(\cdot )\tau \left( \cdot \right)
/r\sigma }$ the last norm is bounded by%
\begin{equation*}
\Big\|\mathcal{M}\Big(\frac{\left( g_{v,k,v_{P}}\right) ^{\frac{\sigma }{%
\tau \left( \cdot \right) }}}{|{Q(\cdot ,2^{k-v_{P}+3})}|^{r\sigma }}\Big)%
\Big\|_{\frac{p(\cdot )\tau \left( \cdot \right) }{r\sigma }}\lesssim \Big\|%
\frac{\left( g_{v,k,v_{P}}\right) ^{\frac{\sigma }{\tau \left( \cdot \right) 
}}}{|{Q(\cdot ,2^{k-v_{P}+3})}|^{r\sigma }}\Big\|_{\frac{p(\cdot )\tau
\left( \cdot \right) }{r\sigma }}.
\end{equation*}%
This term is bounded by $1$ if and only if%
\begin{equation*}
\Big\|\frac{\left( g_{v,k,v_{P}}\right) ^{1/r}}{|{Q(c_{P},2^{k-v_{P}+3})}%
|^{\tau \left( \cdot \right) }}\Big\|_{p(\cdot )}\lesssim 1,
\end{equation*}%
wich follows since $|{Q(\cdot ,2^{k-v_{P}+3})}|\geq 1$. Now, with $%
w>n+c_{\log }(1/q)+c_{\log }\left( \tau \right) $, the term in {$\mathrm{%
\eqref{new-key-est}}$, with }$\dsum\limits_{v=k+4}^{\infty }$ in place of $%
\dsum\limits_{v=0}^{\infty }$ is bounded by 
\begin{eqnarray*}
&&\frac{c\text{ }2^{(w-\frac{n\sigma }{t^{+}})k}}{2^{(w-\frac{n\sigma }{t^{+}%
})k}+2^{kns}}\Big\|\Big(\dsum\limits_{v=k+4}^{\infty }\Big(\frac{\eta
_{v,w}\ast \left( g_{v,k,v_{P}}\right) ^{\frac{\sigma }{\tau \left( \cdot
\right) }}}{|{Q(\cdot ,2^{k-v_{P}+3})}|^{r\sigma }}\Big)^{\frac{q\left(
\cdot \right) \tau \left( \cdot \right) }{r\sigma }}\Big)^{\frac{\sigma r}{%
q\left( \cdot \right) \tau \left( \cdot \right) }}\Big\|_{\frac{p(\cdot
)\tau \left( \cdot \right) }{r\sigma }} \\
&&+\frac{c\text{ }2^{kns}}{2^{(w-\frac{n\sigma }{t^{+}})k}+2^{kns}} \\
&\lesssim &1,
\end{eqnarray*}%
by Theorem \ref{DHR-theorem}. Therefore our estimate {$\mathrm{%
\eqref{principal-est}}$ follows by taking }$0<s<\frac{w}{n}$ and the fact
that{\ }%
\begin{equation*}
d>n(r+1)\tau ^{+}+n+a+w.
\end{equation*}

\textit{Case 2.} $|P|<1$. We have{\ }%
\begin{equation*}
\left\vert P\right\vert ^{-\tau (x)}\leq c\left\vert P\right\vert ^{-\tau
(y)}(1+2^{v_{P}}\left\vert x-y\right\vert )^{c_{\log }\left( \tau \right)
}\leq c\left\vert P\right\vert ^{-\tau (y)}(1+2^{v}\left\vert x-y\right\vert
)^{c_{\log }\left( \tau \right) }
\end{equation*}
for any $x,y\in \mathbb{R}^{n}$. Hence,%
\begin{equation*}
\frac{\eta _{v,w}\ast g_{v,k,v_{P}}}{|P|^{\tau (\cdot )}}\lesssim \eta
_{v,w-c_{\log }\left( \tau \right) }\ast \left( \frac{g_{v,k,v_{P}}}{%
|P|^{\tau (\cdot )}}\right) .
\end{equation*}%
Therefore, we can use the similar arguments above to obtain the desired
estimate, where we did not need to use Lemma \ref{DHHR-estimate}, which
could be used only to move $\left\vert P\right\vert ^{\tau (\cdot )}$ inside
the convolution and hence the proof is complete.

\end{document}